\documentclass[onefignum,onetabnum]{siamart190516}

\usepackage{graphicx}
\usepackage{subcaption}
\usepackage{tikz}
\usepackage{pgfplots}

\newcommand{\rR}{\mathbb{R}}

\newcommand{\eps}{{\varepsilon }}

\newcommand{\cH}{\mathcal{H}}

\newcommand{\cI}{\mathcal{I}}

\newcommand{\rT}{\mathbb{T}}
\newcommand{\tu}{\tilde u}
\newcommand{\tw}{\tilde w}
\newcommand{\pt}{\partial_t}
\newcommand{\ptt}{\partial_{tt}}

\usepackage{lipsum}
\usepackage{amsfonts}
\usepackage{graphicx}
\usepackage{epstopdf}
\usepackage{algorithmic}
\ifpdf
  \DeclareGraphicsExtensions{.eps,.pdf,.png,.jpg}
\else
  \DeclareGraphicsExtensions{.eps}
\fi

\newsiamremark{remark}{Remark}
\newsiamremark{hypothesis}{Hypothesis}
\crefname{hypothesis}{Hypothesis}{Hypotheses}
\newsiamthm{claim}{Claim}

\headers{Wave maps into spheres}{J. Giesselmann, E. M{\"a}der-Baumdicker, D. J. Stonner}

\title{A posteriori error estimates for wave maps into spheres\thanks{Submitted to the editors DATE.
\funding{J.G.\ thanks
 the German Research Foundation (DFG) for support of the project via DFG grant GI 1131/1-1. E. M.-B. is funded by the DFG via the grant MA 7559/1-1 and appreciates the support}}}

\author{Jan Giesselmann\thanks{Technical University of Darmstadt, Department of Mathematics, Dolivostr. 15, 64293 Darmstadt, Germany, 
  (\email{giesselmann@mathematik.tu-darmstadt.de}),}
\and Elena M{\"a}der-Baumdicker\thanks{Technical University of Darmstadt, Department of Mathematics,  Schlossgartenstr. 7,
 64289 Darmstadt, Germany }
  (\email{maeder-baumdicker@mathematik.tu-darmstadt.de}),
  \and
  David Jakob Stonner\thanks{Technical University of Darmstadt}(\email{davidjakob.stonner@stud.tu-darmstadt.de})
}

\usepackage{amsopn}

\ifpdf
\hypersetup{
  pdftitle={A posteriori error estimates for wave maps into spheres},
  pdfauthor={J. Giesselmann and E. Mäder-Baumdicker and D.J. Stonner}
}
\fi

\begin{document}

\maketitle

\begin{abstract}
 We provide a posteriori error estimates in the energy norm for temporal semi-discretisations of wave maps into spheres that are based on the angular momentum formulation. Our analysis is based on novel weak-strong stability estimates which we combine with suitable reconstructions of the numerical solution. We present time-adaptive numerical simulations based on the a posteriori error estimators for solutions involving blow-up.
\end{abstract}

\begin{keywords}
  wave maps, weak-strong stability, a posterirori error estimates, blow-up.
\end{keywords}

\begin{AMS}
65M16, 35L71 , 35B44
	
\end{AMS}

 \section{Introduction} 
 This paper is concerned with the numerical approximation of  wave maps, i.e.,  semi-linear wave equations with the point-wise constraint that the solution takes values in some given target manifold.
 They arise as critical points of a Lagrange functional for manifold valued functions and serve as model problems in general relativity \cite{Aitchison_Hey_2002} and in particle physics \cite{Carroll_2019}.
 We refer to \cite{Shatah_Struwe_1998, Tataru_2004, Krieger_2008} and the introduction in \cite{Rodnianski_Sterbenz_2010} for an overview on the general theory of wave maps. The monograph \cite{Geba_Grillakis_2017} contains a detailed introduction to the recent development in the analysis of wave maps.

There are two key challenges in numerically approximating wave maps:
 One is the point-wise constraint that make the function spaces in which solutions are sought non-linear and, the second,
 is gradient-blow up that leads to highly localized phenomena in space and time that need to be 
 suitably resolved by numerical methods.
 A variety of different numerical methods that deal with the point-wise constraint using different approaches such as projections, penalties and Lagrange multipliers has been proposed 
 \cite{
 Bartels_2009, Bartels_2016, Bartels_Feng_Prohl_2007, Bartels_Lubich_Prohl_2009, Bizon_Chmaj_Tabor_2001, Cohen_Verdier_2016, Kycia_2012}
 and for the methods of Bartels and coworkers a priori convergence analysis is available, in the sense that stability estimates are proven that imply convergence of subsequences to weak solutions.
 
It seems desirable to obtain more quantitative information on the accuracy of numerical approximations and,
given the highly localized dynamics of gradient blow-up,  we aim to provide a posteriori error estimates.
We will focus on a scheme whose a priori analysis was studied in \cite{Bartels_2015} and \cite{Karper_Weber_2014}.
For the schemes at hand, convergence results are available even beyond gradient-blow up, i.e.\ limits of subsequences of numerical solutions are weak solutions,
but quantitative estimates beyond singularity formation seem to be out of reach due to discontinuity of the solution operator \cite{DAncona_Georgiev_2004}, see Section 2 for details.
Thus, we focus on estimates for the solution up to the blow up time.

Similarly to what was done in \cite{Georgoulis_Lakkis_Makridakis_Virtanen_2016, Kyza_Makridakis_2011}, we study errors entering via temporal discretization and, indeed,
we restrict our study to semi-discretization in time. 
While the development of estimators for spatial discretization errors is certainly an important task in its own right, it is beyond the scope of this work.
It will probably require a suitable extension of elliptic reconstruction techniques to harmonic maps, that does not seem available yet. 
Indeed, little has been proven concerning convergence of numerical schemes for harmonic maps. In certain situations, uniqueness and regularity of hamonic maps can be ensured, which allows to show high order convergence of so-called geodesic finite elements
\cite{GrohsHarderingSander2015}.
In the general case, only existence of weak harmonic maps is guaranteed and for general triangulations and minimal regularity solutions, an a posteriori criterion is needed in order to guarantee weak convergence of numerical solutions \cite{Bartels2005}.
To the best of our knowledge no quantitative a posteriori error bounds  are available for numerical approximation schemes for harmonic maps and harmonic map heat flows.

For a long time, a posteriori error control for (linear) wave equations  was limited to first order schemes \cite{Bernardi_Suli_2005, Georgoulis_Lakkis_Makridakis_2013}.
Earlier works on adaptivity for wave equations can be found in \cite{Adjerid_2002,Bangerth_Rannacher_2001, Suli_1996}.
Quite recently, a posteriori error estimates for second order multi-step time discretisations of the linear wave equation were derived \cite{Georgoulis_Lakkis_Makridakis_Virtanen_2016}.

Due to the appearance of singularities, our goals are similar to those pursued in \cite{Cangiani_Georgoulis_Kyza_Metcalfe_2016,Kyza_Makridakis_2011} that study blow-up for semi-linear parabolic equations.
We also refer to these papers for earlier works on numerical approximation of blow-up solutions of nonlinear PDEs such as nonlinear Schr\"odinger equations or semi-linear parabolic problems.
Let us mention that the blow-up mechanism in wave maps is rather different from the blow-up mechanisms in  nonlinear Schr\"odinger equations or semi-linear parabolic problems
where the $L^\infty$-norm blows up in finite time.

In order to derive the desired a posteriori error estimates, we use two main ingredients:
Firstly, a suitable reconstruction of the numerical solution that can be understood as the exact solution of a perturbed version of the \emph{angular momentum formulation} for wave maps into spheres and, secondly, a novel weak-strong stability principle. 
One key feature of our reconstruction is to ensure that the estimator is formally of optimal order, i.e., that it converges to zero with the same rate as the true error on equidistant meshes.


The remainder of this work is organized as follows: We review some facts from the analysis of wave maps in Section 2 and introduce the problem and basic notation in Section 3. Section 4 provides two stability estimates, one is based on a first order reformulation of the problem that is  available when the target manifold is $S^2$ and the other Theorem covers the general case.
In Section 5, we provide a posteriori error estimates for a numerical scheme that is based on the first order reformulation of the problem. The main contribution of this section is the construction of suitable reconstructions of the numerical solution; whereas computable bounds for the residuals, that appear when the reconstruction is inserted into the wave map problem, are postponed to the appendix.
Finally, in Section 6, we report on numerical experiments using adaptive time stepping based on the a posteriori error estimators derived before.

\section{Background on wave maps}

A specific feature of wave maps is that depending on the size of initial data (in suitable Sobolev norms) and the dimension
 of the target manifold either strong solutions may exist on arbitrarily long time intervals or solutions may exhibit gradient blow-up in finite time, see \cite{Rodnianski_Sterbenz_2010}. But there are also  larger classes of solutions, namely \textit{distributional} or \textit{weak solutions}, and in particular finite energy weak solutions. 
 Note that weak solutions for example exist as accumulation points of subsequences of numerical schemes in \cite{Bartels_2015}.  Existence of global weak solutions for finite energy initial data in $2+1$ dimensions was established in \cite{Mueller_Struwe_1996}.
 We will introduce our precise notion of finite energy weak solution in the next section. If a (global) finite energy weak solution and a non-global strong solution exist, then uniqueness results as in \cite{Struwe_1999, Widmayer_2015} guarantee that both solutions agree until the appearance of the singularity. This is, under certain circumstances, strong solutions can be extended as weak solutions through the singularity. Conditions such as an energy inequality are needed to get this uniqueness result because, in general, weak solutions are not unique \cite{Widmayer_2015}. 
 
 We provide weak-strong stability results Theorems \ref{thm:stab} and \ref{thm:stabN} that can be seen as  more quantitative versions of the (finite energy) weak-strong-uniqueness results in \cite{Widmayer_2015,Struwe_1999}. Indeed,  Theorems \ref{thm:stab} and \ref{thm:stabN} assert that
as long as there exists  a strong solution for certain initial data, there are explicit bounds for the difference between this solution and solutions to perturbed problems even if those are only 
finite energy weak solutions.  Note that this quantitative control requires the same regularity as the uniqueness result \cite{Struwe_1999}, see Remark \ref{rem:regs} for more details.

We employ Theorem  \ref{thm:stab} (that assumes that the target manifold is $S^2$) in our a posteriori error analysis but
since weak-strong stability results are interesting in their own right, we present a weak-strong stability result for general target manifolds in Theorem \ref{thm:stabN}. While it can be thought of as an extension of Theorem \ref{thm:stab} (where the target manifold is $S^2$) there are some significant differences on the technical level that will be discussed in Remark~\ref{rem:stabdiff} and Remark~\ref{rem:stabdiff1B}. Those are the reason why we base our a posteriori error analysis on Theorem \ref{thm:stab}. This is discussed in more detail in Remark \ref{rem:stabdiff2}.

  Uniqueness and stability properties of wave maps do not only depend on the target manifold but also on the dimension of the domain. Uniqueness  for wave maps in $1+1$ dimensions was shown in \cite{Zhou_1999} and non-uniqueness in the supercritical dimension $3+1$ was shown in 
 \cite{Shatah_Tahvildar-Zadeh_1994, Widmayer_2015}. To the best of our knowledge, in the critical dimension $2+1$ uniquenss of weak solutions to finite energy data is unknown and it is unclear whether imposing an energy inequality restores uniqueness in $2$ or more space dimensions.
 An interesting observation in $2+1$ dimensions is that solutions (even if they are unique) do not depend continuously on the initial data in the energy norm \cite{DAncona_Georgiev_2004}. It should be noted that our weak-strong stability results use the energy norm and are valid in arbitrarily many space dimensions.
 It follows from \cite{DAncona_Georgiev_2004} that any such stability result, i.e.\ any bound for the difference of two solutions, measured in the energy norm, needs to involve a stronger norm (than the energy norm) of at least one of the solutions. This is a fundamental obstacle to deriving a posteriori error estimates (in the energy norm) that are convergent (i.e.\ go to $0$ for $\tau, h \searrow 0$, where $\tau,\, h$ denote spatial and temporal mesh width respectively) in case the exact solution does not have any additional regularity (beyond being a finite energy weak solution).

 \section{Problem statement and notation}
 For some bounded Lipschitz domain $\Omega \subset \rR^m$,  some final time $\rT>0$ and  some $n$-dimensional submanifold without boundary $N \subset \rR^\ell$ a wave-map is a map
 \begin{equation}\label{wavemap}
  u : (0,\rT) \times \Omega \rightarrow N,\ \text{ satisfying } \quad 
  \partial_t^2 u - \Delta u  \perp T_uN  \quad \text{in } (0,\rT) \times \Omega
 \end{equation}
where $T_uN$ denotes the tangent space of N at $u$.
Equation \eqref{wavemap} needs to be complemented with initial and boundary data.
To this end, maps $u_0 : \Omega \rightarrow N$ and $u_1 : \Omega \rightarrow \rR^\ell$ such that $u_1(x) \in T_{u_0(x)}N$ for all $x \in \Omega$ are fixed and one requires
\begin{equation}\label{inic}
 u(0,\cdot)=u_0, \quad \partial_t u(0,\cdot) = u_1,
\end{equation}
and homogeneous Neumann boundary conditions 
\begin{equation}\label{boundary}
\partial_n u  =0 \quad \text{ on } (0,\rT) \times \partial \Omega.
\end{equation}
Strong solutions of the wave map equation satisfy an energy conservation principle 
\begin{equation}
 E[u(t),\partial_t u(t)]:= 
 \frac 12 \int_\Omega |\partial_t u(t)|^2 + |\nabla u(t)|^2 \, dx =E[u_0,u_1]
\end{equation}
and are critical points of the Lagrangian
\begin{equation}
 L[u(t),\partial_t u(t)]:= 
 \frac 12 \int_0^{\mathbb{T}} \int_\Omega |\partial_t u(t)|^2 - |\nabla u(t)|^2 \, dx.
\end{equation}
We reformulate the wave map equation (\ref{wavemap}) in order to see that this is a semi-linear wave equation.
Let $A_p(\cdot,\cdot): T_pN\times T_pN\to (T_pN)^\perp$ be the second fundamental form of the compact submanifold $N$ at a point $p\in N$. We denote the variables on $[0,\rT)\times\Omega$ by $(t,x) = (x^\alpha), 0\leq \alpha\leq m$. We raise and lower indices with the Minkowski metric $(\eta_{\alpha\beta}) = \operatorname{diag}(-1,1,...,1)$ and we sum over repeated indices.  Then a (strong) wave map is a map $u=(u^1,...,u^\ell):[0,\rT)\times\Omega\to N \hookrightarrow \rR^\ell$ that satisfies
\begin{align}\label{wms0}
 \partial_t^2u-\Delta u = A[u](Du,Du),
\end{align}
where $A[u](Du,Du)$ stands for $\left(A^i_{jk}\big|_u\partial_\alpha u^j\partial^\alpha u^k\right)\big|_{1\leq i\leq \ell}$, see \cite{Shatah_Struwe_1998}.\\

A significant part of our analysis will consider the case that the target manifold $N$ is the $2$-sphere $S^2 \subset \rR^3$ and in this case \eqref{wms0} reduces to
\begin{align}
 \partial_t^2 u - \Delta u  = (|\nabla u|^2- |\partial_t u|^2) & u \label{wms}\\
 \text{with point-wise constraint } |u(t,x)| &=1.
\end{align}
Let us also mention that, using angular momentum $\omega:=\partial_t u \times  u,$  the wave map equation can be phrased as \cite{Karper_Weber_2014}
\begin{equation}\label{wmom}
\begin{aligned}
 \partial_t u = u \times \omega\ \ \ \ \text{ and }\  \ \ \ \partial_t \omega = \Delta u \times u.
\end{aligned}
\end{equation}
This variant is the one underlying the numerical scheme that we will study.


 \section{A quantitative stability estimate}
 In this section, we establish a weak-strong stability result that complements weak-strong uniqueness results in \cite{Widmayer_2015,Struwe_1999} by providing bounds for differences between solutions. In particular, our estimates quantify the the impact of residuals that is crucial for the use of stability results in proving a posteriori error estimates.
 We prove weak-strong stability in the general case as well as in the special case of $N=S^2$ since we believe that the former nicely highlights the general geometric structure while the latter proof uses very elementary techniques and does not require any background in differential geometry.\\[-0.2cm]

 We denote by $(\cdot, \cdot)$ the $L^2$ inner product on $(0,\rT) \times \Omega$ and by $(\cdot, \cdot)_\Omega$ the $L^2$ inner product on $ \Omega$.
  \begin{definition}\label{def:weak}
 Given $u_0 \in H^1(\Omega,N)$ and $u_1 \in L^2(\Omega, \mathbb{R}^\ell)$ so that $u_1(x) \in T_{u_0(x)}N$ for almost all $x$, 
 we call a function  $u \in L^2([0,\rT) \times \Omega)$ with values in $N$ and $\nabla u, \partial_t u \in   L^\infty([0,\rT), L^2(\Omega))$ a \emph{finite energy weak solution} of \eqref{wavemap}-\eqref{boundary} provided the following conditions are satisfied
  \begin{enumerate}
   \item $u(0,\cdot)=u_0$
   \item $(\partial_t u , \partial_t \psi) - (\nabla u , \nabla \psi)= - \big(A[u](Du,Du),\psi\big) - (u_1, \psi(0,\cdot))_\Omega$ \\for all $ \psi \in C_c^\infty ([0,\rT)\times \bar \Omega)$
   \item  $ E[u(t),\partial_t u(t)] \leq  E[u_0,u_1]$ for almost all $0<t<\rT$.
  \end{enumerate}
  We say that $u$ satisfies a \textit{local}  energy condition provided
  \[  E[u(t),\partial_t u(t)] \leq  E[u(s),\partial_t u(s)] \text{ for almost all } 0<s < t<\rT\]
 \end{definition}

 \begin{remark}\label{rem:weak}
 Note that the weak formulation in Definition \ref{def:weak} actually holds for all $\psi \in  W^{1,1}(0,\rT;L^2(\Omega)) \cap L^1(0,\rT; L^\infty \cap H^{1}(\Omega))$ with $\psi(\rT,\cdot)=0,$ due to a density argument where pass  to the limit in $\big(A[u](Du,Du),\psi\big)$ using majorised convergence, see the appendix of \cite{FMS1998} for details.
 \end{remark}
 
 We are going to compare a finite energy weak solution $u$ to a strong solution $(\tu,\tw)$ of the perturbed problem
 \begin{align}\label{perturbed}
\partial_t \tu = \tu \times \tw + r_u  , \qquad 
\partial_t \tw = \Delta \tu \times \tu +r_w.
 \end{align}

\begin{lemma}\label{2801} For any two sufficiently regular functions $\tu:\Omega \rightarrow S^2,\ \tw : \Omega \rightarrow \rR^3$ the following identities hold:
 \begin{enumerate}
  \item[(a)] $(\tu \times \tw ) \times \tw = (\tu \cdot \tw)\tw - |\tw|^2  \tu  $\\[-3mm]
  \item[(b)] $\tu \times (\Delta \tu \times \tu) = \Delta \tu + |\nabla \tu|^2 \tu$
 \end{enumerate}

\end{lemma}

\begin{proof}
 Part (a) immediately follows from $(a \times b) \times c = (a \cdot c) b - (b \cdot c) a $ for any $a,b,c \in \mathbb{R}^3$.
 In order to prove (b), we note that  $(\partial_{x_j} \tu )\cdot \tu=\frac 12 \partial_{x_j} |\tu|^2=0$ so that
 \begin{align*}
  \tu \times (\Delta \tu \times \tu) &= \sum_{j=1}^3 [ \partial_{x_j} ( \tu \times ( \partial_{x_j}\tu \times \tu)) - (\partial_{x_j} \tu) \times (\partial_{x_j} \tu \times \tu)]\\
  &= \sum_{j=1}^3 [ \partial_{x_j} ( |\tu|^2   \partial_{x_j}\tu  - (\partial_{x_j}\tu  \cdot \tu) \tu) - (\partial_{x_j} \tu \cdot \tu)\partial_{x_j}\tu +   |\partial_{x_j} \tu |^2 \tu]\\
  &= 
  \sum_{j=1}^3 [ \partial_{x_j} (    \partial_{x_j}\tu )+   |\partial_{x_j} \tu |^2 \tu].
 \end{align*}

\end{proof}

 \begin{theorem}\label{thm:stab}
 Let $u$ be a  finite energy weak solution of \eqref{wavemap}-\eqref{boundary} satisfying the local energy condition. Let $p=m$ for $m\geq 3$ and $p\in(2,\infty]$ for $m=2$. Consider $(\tu,\tw)$, a solution of the perturbed problem \eqref{perturbed}. We ask for the following regularity: $\tu \in L^\infty((0,\rT)\times \Omega,S^2)$ and
 \begin{align*}
  \nabla \tu, \tw\in & \Big( L^1((0,\rT); H^1(\Omega))   \cap L^\infty ((0,\rT); L^2 (\Omega)) \\
  &\qquad \cap L^k((0,\rT); L^\infty(\Omega))\cap L^{\frac{k}{k-1}}((0,\rT); L^p(\Omega))\Big)
 \end{align*}
 for a $k\in[1,2]$ and with given functions $r_w \in L^1((0,\rT); L^2( \Omega))$ and \\
 $r_u \in L^1((0,\rT);H^1( \Omega))\cap L^{\frac{k}{k-1}}((0,\rT); L^2(\Omega)).$
 Then, the difference at time $t$ can be controlled via the difference in initial data and perturbation terms. Indeed, $\cH$ defined by
  \[\cH(t):=  \frac 12 \int_\Omega |\partial_t u(t) - \tu(t) \times \tw(t) |^2 + |\nabla u(t) - \nabla \tu(t)|^2 +|u(t) -\tu(t)|^2\, dx,\]
  satisfies
 \begin{align}\label{eq:stab_loc}
  \sqrt{\cH(t)} \leq
  \left(   \sqrt{\cH(s)} +  \left( \int_s^t \alpha(\tau )  \, d\tau\right) \right)  \times \exp\left(\frac 12\int_s^t \delta(\tau)  \, d\tau \right)
  \end{align}
for almost all $s<t \in (0,\rT)$ with
\begin{align*}
 \alpha&:=  \| r_g   + r_u \times \tw + \tu \times r_w \|_{L^2(\Omega)} + \| r_u \|_{L^2(\Omega)} + \| \nabla r_u \|_{L^2(\Omega)} \\
  r_g &:= (\tu \cdot \tw)\tw - |\tu \cdot \tw|^2  \tu;\quad  \mathcal{A}[\tu]:= |\nabla \tu|^2 - | \tu \times \tw|^2\\
  \delta&:= \left( 1+  c_q \| \mathcal{A}[\tu]\|_{L^p( \Omega)} + 2c_q \|\tu \times \tw\|_{L^{2p}( \Omega)}
  ^2\right. \\ 
  &\qquad \ \left. + 2c_q \| \nabla \tu\|_{L^{2p}( \Omega)} \|\tu \times \tw \|_{L^{2p}( \Omega)} + 4 \|\tu \times \tw \|_{L^\infty( \Omega)}\right)
\end{align*}
where $c_q$ is the squared constant of the Sobolev embedding $H^1(\Omega) \rightarrow L^{\tfrac{2p}{p-2}}(\Omega)$.
This implies that for almost any sequence $0 =t_0 < t_1 < \dots < t_N \leq \rT$ 
\begin{multline}\label{eq:stab_glo}
  \sqrt{\cH(t_N)} \leq
    \sqrt{\cH(0)}\exp\left(\frac 12\int_0^{t_N} \delta(\tau)  \, d\tau \right) \\
    +  \sum_{j=1}^N\left( \int_{t_{j-1}}^{t_j} \alpha(\tau )  \, d\tau\right)  \times \exp\left(\frac 12 \int_{t_{j-1}}^{t_N} \delta(\tau)  \, d\tau \right)
  \end{multline}

\end{theorem}

\begin{remark}[Constants in the estimate]\label{rem:sobcon}
Explicit upper bounds for $c_q$, for a variety of domains, can be found in \cite{MizuguchiTanakaSekineOishi2017}.
\end{remark}

\begin{remark}[Regularity of weak solution]\label{rem:regw}
 Note that if $u$ is a finite energy weak solution, not necessarily satisfying the local energy condition, then \eqref{eq:stab_loc} still holds in the special case $s^*=0$.
\end{remark}

\begin{remark}[Regularity of strong solution]\label{rem:regs}
  Note that our regularity assumptions with $k=1$ correspond to the conditions in Struwe's result \cite[Thm~2.2]{Struwe_1999}. The case studied in  \cite[Thm~2.2]{Struwe_1999} corresponds to $m=2$ and $r_u=r_\omega=r_g=0$ in our notation so that $D\tu \hat= (\tu \times \tw, \nabla \tu)$.
%
%
  Thus, choosing $k=1$, $p=2+\epsilon$, we assume $ \nabla \tu, \tw\in  \Big( L^1((0,\rT); H^1(\Omega)\cap L^\infty(\Omega))    
  \cap L^{\infty}((0,\rT); L^{2+\epsilon}(\Omega))\Big)=:X$.\\  
  Struwe assumes $D\tu\in X$ which also holds under our assumptions mainly because $\tu\in L^\infty( (0,\rT)\times\Omega)$. There is the difference that we consider Neumann boundary data and Struwe has a solution on full $\mathbb{R}^2$.\\[-0.2cm]
  
  Note furthermore that the conditions in our theorem allow us to use $\tu \times \tw$ as test function in $(2)$ of Definition~\ref{def:weak}, i.e.\ $\tu\times\tw \in W^{1,1}((0,\rT);L^2(\Omega))\cap L^1((0,\rT);L^\infty\cap H^1(\Omega))$. This can be seen as follows:
  First, we observe that $\tu\times \tw \in L^1((0,\rT);L^\infty(\Omega))$ and $\nabla(\tu\times\tw)\in L^1((0,\rT);L^2(\Omega)) $ using Youngs inequality $\int \|\nabla\tu\|_{L^p(\Omega)} \|\tw\|_{L^q(\Omega)} dt\leq c\int \|\nabla\tu\|_{L^p(\Omega)}^{\frac{k}{k-1}} + \|\tw\|^k_{L^q(\Omega)} dt<\infty$ ($\frac{1}{p} + \frac{1}{q}=\frac{1}{2}$). It remains to check $\partial_t(\tu\times\tw)\in L^1((0,\rT); L^2(\Omega))$:
   Using Lemma \ref{2801}, we know that
\begin{align}
\begin{split}\label{2801b}
 \partial_t (\tu \times \tw)&=  \partial_t \tu \times \tw +  \tu \times \partial_t \tw \\
 &= (\tu \times \tw) \times \tw + r_u \times \tw +  \tu \times (\Delta \tu \times \tu) + \tu \times r_w
 \\
 &= r_g + ( |\nabla \tu|^2 - |\tu \times \tw|^2) \tu  + r_u \times \tw + \Delta \tu+ \tu \times r_w,
 \end{split}
\end{align}
where we have used $|\tu \times \tw|^2 = |\tw |^2 - |\tw \cdot \tu|^2.$ Note that $r_g\sim |\tw|^2$. We always have that $\tw, \nabla\tu\in L^2((0,\rT);L^4(\Omega))$ because of
\begin{align*}
 \int \|\tw\|_{L^4}^2 dt \leq \int \|\tw\|_{L^\infty(\Omega)} \|\tw\|_{L^2} dt\leq \begin{cases}
        \text{esssup}_{t}\|\tw\|_{L^2(\Omega)} \int \|\tw\|_{L^\infty} dt <\infty &\text{ if } k=1\\
        c\int \|\tw\|_{L^\infty(\Omega)}^k + \|\tw\|_{L^2(\Omega)}^{\frac{k}{k-1}} dt <\infty & \text{ if } k>1.
       \end{cases}
\end{align*}
%
%
The conditions on $r_u$ imply that also $r_u\times\tw\in L^1((0,\rT);L^2(\Omega))$. Thus, all terms in (\ref{2801b}) are controlled in $L^1((0,\rT);L^2(\Omega))$.

\end{remark}

\begin{proof}
 Let us fix some (arbitrary) $s^*<t^* \in (0,\rT)$ and let us define for any $0 < \eps < \min\{\rT-t^*, t^* - s^*\}$ the map 
 \[ \phi_\eps (t):= \left\{ \begin{array}{ccc}
                           0 &:& t < s^*\\
                           \frac{t - s^*}{\varepsilon} &:&  s^* < t < s^* + \varepsilon\\
                           1 &: & s^* + \varepsilon \leq t\leq t^* \\ 
                             1 - \frac{t-t^*}{\eps}&: & t^* \leq t\leq t^* + \eps\\
                             0 &:&  t^* + \eps < t
                            \end{array}\right.
\]
We will study $\lim_{\eps \searrow 0 } \int_0^\rT \cH(t) \partial_t  \phi_\eps (t) \, dt $.
On the one hand, for every pair of Lebesgue points of $t \mapsto  E[u(t),\partial_t u(t)] $ we have
\begin{equation}\label{cH0}
\lim_{\eps \searrow 0 } \int_0^\rT \cH(t) \partial_t  \phi_\eps (t) \, dt  = - \lim_{\eps \searrow 0  } \frac{1}{\eps}  \left( \int_{t^*}^{t^*+\eps} \cH(t)   \, dt 
- \int_{s^*}^{s^*+\eps} \cH(t)   \, dt\right)= - \cH(t^*) + \cH(s^*).
\end{equation}
On the other hand, we may decompose the integral at hand as
\begin{align}\begin{split}
 \int_0^\rT \cH(t) \partial_t & \phi_\eps (t) \, dt  =
  \int_0^\rT E[u(t),\partial_t u(t)]  \partial_t  \phi_\eps (t) \, dt \\
  &\ - \int_0^\rT ( \partial_t u \cdot \tu \times \tw  + \nabla u \cdot \nabla \tu) \partial_t  \phi_\eps (t) \, dt \\
 &  \ + \int_0^\rT E[\tu(t),\tu(t)\times \tw(t)]  \partial_t  \phi_\eps (t) \, dt + \int_0^\rT  \frac 12 |u(t) -\tu(t)|^2  \partial_t  \phi_\eps (t) \, dt\\
  &=: E^1_\eps - E^2_\eps +E^3_\eps + E^4_\eps.
  \end{split}
\end{align}
Concerning $E^1_\eps$, we observe that for every pair of Lebesgue points of $t \mapsto  E[u(t),\partial_t u(t)] $ 
\begin{multline}\label{eq:e1e}
\lim_{\eps \searrow 0 } E^1_\eps = \lim_{\eps \searrow 0 } \int_0^\rT E[u(t),\partial_t u(t)] \partial_t  \phi_\eps (t) \, dt  \\
= -\lim_{\eps \searrow 0  }  \frac{1}{\eps}\left( \int_{t^*}^{t^*+\eps} E[u(t),\partial_t u(t)] \, dt - \int_{s^*}^{s^*+\eps} E[u(t),\partial_t u(t)] \, dt\right)\\=  - E[u(t^*),\partial_t u(t^*)]+ E[u(s^*),\partial_t u(s^*)] \geq 0.
\end{multline}
Next, we consider $E^2_\eps$ for fixed $\eps$ and observe
\begin{equation}\label{eq:e2e0}
 \begin{aligned}
  E^2_\eps &= ( \partial_t u, \tu \times \tw \partial_t \phi_\eps ) + (\nabla u, \nabla \tu \partial_t \phi_\eps)\\
    &= ( \partial_t u, \partial_t( \tu \times \tw  \phi_\eps) ) -( \partial_t u, \partial_t (\tu \times \tw) \phi_\eps )+ (\nabla u, \nabla \tu \partial_t \phi_\eps)\\
    &=  ( \nabla u, \nabla( \tu \times \tw  \phi_\eps) )  - ((|\nabla u|^2 - |\partial_t u|^2)u,\tu \times \tw  \phi_\eps) 
    \\
    & \quad -( \partial_t u, \partial_t (\tu \times \tw) \phi_\eps )+ (\nabla u, \nabla \tu \partial_t \phi_\eps).
 \end{aligned}
\end{equation}
We also note that
\begin{align}
 \begin{split}
\label{2801c}
 ( \nabla u, \nabla( \tu \times \tw  \phi_\eps) )+  (\nabla u, \nabla \tu \partial_t \phi_\eps)
 &= ( \nabla u, \nabla( \partial_t \tu \phi_\eps) )- ( \nabla u,\nabla r_u\phi_\eps)\\
 & \hspace{4cm}+  (\nabla u, \nabla \tu \partial_t \phi_\eps)\\
 &=( \nabla u, \nabla \partial_t( \tu \phi_\eps) )- ( \nabla u,\nabla r_u\phi_\eps).
 \end{split}
\end{align}
We insert \eqref{2801b} and  \eqref{2801c} into \eqref{eq:e2e0} and obtain 
\begin{equation}\label{eq:e2e}
 \begin{aligned}
  E^2_\eps 
   & =  ( \nabla u, \nabla \partial_t( \tu \phi_\eps) )- ( \nabla u,\nabla r_u\phi_\eps) - ((|\nabla u|^2 - |\partial_t u|^2)u,\tu \times \tw  \phi_\eps) \\
    & 
    \quad -( \partial_t u, (r_g  + ( |\nabla \tu|^2 - |\tu\times \tw|^2) \tu  + r_u \times \tw + \Delta \tu+ \tu \times r_w)\phi_\eps)\\
    &= 
    - ( \nabla u,\nabla r_u\phi_\eps) - ((|\nabla u|^2 - |\partial_t u|^2)u,\tu \times \tw  \phi_\eps) \\
    & \quad  
    -( \partial_t u, (r_g + ( |\nabla \tu|^2 - | \tu \times \tw|^2) \tu  + r_u \times \tw + \tu \times r_w)\phi_\eps)
 \end{aligned}
\end{equation}
where we have used integration by parts in the last equality.
Equation \eqref{eq:e2e} allows us to conclude
\begin{multline}\label{eq:e2f}
\lim_{\eps \searrow 0 } E^2_\eps = 
 - \int_{s^*}^{t^*}\! \int_\Omega \big\{ \partial_t u \cdot[ r_g  + ( |\nabla \tu|^2 - |\tu\times \tw|^2) \tu  + r_u \times \tw + \tu \times r_w]\\
 +  \nabla u\cdot \nabla r_u 
 + (|\nabla u|^2 - |\partial_t u|^2) u \cdot (\tu \times \tw) \big\}\, dx ds.
\end{multline}
Concerning $E^3_\eps$, we use \eqref{2801b} and integration by parts to obtain
\begin{equation}\label{eq:e3e}
 \begin{aligned}
  E^3_\eps &= \frac 12 ( \tu \times \tw , \tu \times \tw \partial_t \phi_\eps ) + \frac 12 (\nabla \tu, \nabla \tu \partial_t \phi_\eps)\\
    &= 
     \ -( \tu \times \tw \phi_\eps, r_g  + ( |\nabla \tu|^2 - |\tu \times \tw|^2) \tu  + r_u \times \tw + \Delta \tu+ \tu \times r_w ) + ( \partial_t \tu, \Delta \tu  \phi_\eps)\\
    &= 
    -( \tu \times \tw \phi_\eps, r_g    + r_u \times \tw + \tu \times r_w )+ ( r_u, \Delta \tu  \phi_\eps),
 \end{aligned}
\end{equation}
where we have used point-wise orthogonality of $\tu$ to $\tu \times \tw$ in the last equality.
Equation \eqref{eq:e3e} allows us to conclude
\begin{align}\begin{split}
\label{eq:e3f}
\lim_{\eps \searrow 0 } E^3_\eps =  
- \int_{s^*}^{t^*} \int_\Omega (\tu \times \tw)  \cdot[ r_g    + r_u \times \tw + \tu \times r_w ] + \nabla r_u \cdot \nabla \tu \, dx ds.
\end{split}
\end{align}
Finally, we find
\begin{equation}\label{eq:e4f} 
\lim_{\eps \searrow 0} E^4_\eps = 
- \int_{s^*}^{t^*} \int_\Omega (u - \tu) \cdot (\partial_t u - \tu \times \tw - r_u)\, dx dt.
\end{equation}
We combine \eqref{cH0}, \eqref{eq:e1e}, \eqref{eq:e2f}, \eqref{eq:e3f} and \eqref{eq:e4f} to obtain
\begin{align}\begin{split}\label{2801e}
 & \cH(t^*)- \cH(s^*)\\ \leq  & \int_{s^*}^{t^*} \hspace{-0.2cm}\int_\Omega   (  \tu \times \tw- \partial_t u)  \cdot[ r_g    + r_u \times \tw + \tu \times r_w ] - \nabla r_u \cdot (\nabla u - \nabla \tu) dx ds\\ 
  &+\int_{s^*}^{t^*}\hspace{-0.2cm} \int_\Omega (u-\tu)\cdot (\partial_t u-\tu\times\tw) - (u - \tu) \cdot r_u dx ds\\
  &- \int_{s^*}^{t^*}\hspace{-0.2cm} \int_\Omega  ( |\nabla \tu|^2 - |\tu\times \tw|^2) \tu \cdot\partial_t u  + (|\nabla u|^2 - |\partial_t u|^2) u \cdot \tu \times \tw\,  dx ds.
  \end{split}
\end{align}
Let us define $\mathcal{A}[u]:= (|\nabla u|^2 - |\partial_t u|^2)$ and (with a slight abuse of notation) $\mathcal{A}[\tu]:= (|\nabla \tu|^2 - |\tu\times \tw |^2)$. Then,  $u \perp \partial_t u$ and $\tu \perp \tu \times \tw$ allows us to infer
\begin{equation*}
 \partial_t u \cdot \mathcal{A}[\tu] \tu   +  (\tu \times \tw)   \cdot \mathcal{A}[u] u = - (\partial_t u - \tu \times \tw) \cdot \mathcal{A}[\tu] (u - \tu) + \tu \times \tw ( \mathcal{A}[u]-\mathcal{A}[\tu])(u-\tu).
\end{equation*}
This implies, for any pair $p,q >1$ such that $\tfrac 1p + \tfrac 1q = \tfrac 12$,
\begin{equation}\label{2801f}
 \begin{aligned}
  &\int_{s^*}^{t^*} \| \partial_t u \cdot \mathcal{A}[\tu] \tu   +  (\tu \times \tw)   \cdot \mathcal{A}[u] u \|_{L^1( \Omega)}  dt\\
 \leq & \int_{s^*}^{t^*} \|\partial_t u - \tu \times \tw \|_{L^2( \Omega)}\|\mathcal{A}[\tu] \|_{L^p( \Omega)}\|u - \tu\|_{L^q ( \Omega)}\\
 &\quad+ \|(\tu \times \tw) (\mathcal{A}[u] - \mathcal{A}[\tu] ) (u - \tu) \|_{L^1( \Omega)} \, dt\\
 \leq& \int_{s^*}^{t^*} \|\partial_t u - \tu \times \tw \|_{L^2( \Omega)}\Big(  \| \mathcal{A}[\tu]\|_{L^p( \Omega)} 
 + 2 \|\tu \times \tw\|_{L^{2p}( \Omega)}^2\Big) \|u - \tu\|_{L^q (\Omega)}\\
 & \quad + 2 \| \nabla \tu\|_{L^{2p}( \Omega)} \|\tu \times \tw \|_{L^{2p}( \Omega)}
 \|\nabla u - \nabla  \tu\|_{L^2 (\Omega)}\|u - \tu\|_{L^q ( \Omega)} \\ 
 &
 \quad + 2 \|\nabla u - \nabla  \tu\|_{L^2 ( \Omega)}^2  \|\tu \times \tw \|_{L^\infty( \Omega)}
 + 2 \|\partial_t u - \tu \times \tw \|_{L^2( \Omega)}^2  \|\tu \times \tw \|_{L^\infty( \Omega)}\, dt
 \end{aligned}
\end{equation}
where we have used that
\begin{multline} 
\mathcal{A}[u] - \mathcal{A}[\tu]  = 2 \nabla \tu \cdot (\nabla u - \nabla  \tu )   + |\nabla u - \nabla  \tu |^2  \\
- 2 (\tu \times \tw) \cdot (\partial_t u - \tu \times \tw )   - |\partial_t u -  \tu \times \tw |^2 
\end{multline}
and $u, \tu \in S^2$.
We insert  \eqref{2801f} into \eqref{2801e} and obtain, for those $q$ for which $H^1(\Omega)$ embeds into $L^q(\Omega)$
\begin{align}
 \begin{split}
\label{2801g}
  \cH(t^*) &\leq \cH(s^*) + \int_{s^*}^{t^*} 2 \sqrt{\cH(s)} \alpha(s)
  ds+ \int_{s^*}^{t^*}
 \delta(s)\cH(s)  ds\,.
 \end{split}
\end{align}
According to \cite[Thm~21]{Silvestru_2003} equation \eqref{2801g} implies \eqref{eq:stab_loc}. Equation \eqref{eq:stab_glo} follows by taking the square-root of \eqref{eq:stab_loc} and induction in $j$.

\end{proof}

We now come to the computations for a general closed target manifold $N$. Since for general $N$, we use one second order equation instead of a system of first order equations we have only one residual, but we split this into two parts, assuming that one part can be expressed as a time derivative.

\begin{theorem}\label{thm:stabN}
Let $u$ be a weak solution of the problem \eqref{wavemap}-\eqref{boundary} and $\tilde u=(\tilde u^1,...,\tilde u^\ell) :[0,\rT)\times\Omega \to N $ a strong solution of a perturbed problem 
$$\partial^2_t\tilde u - \Delta \tilde u = A[\tilde u](D\tilde u,D\tilde u) + R_1+ \pt R_2, \ \partial_n\tu=0.$$ 
Let $s=m$ for $m\geq 3$ und $s\in(2,\infty]$ for $m=2$. We assume that there is a $k\in[1,\frac{4}{3}]$ such that 
\begin{align*}
 \tu & \in W^{2, 1}((0,\rT) \times \Omega), \\ 
 D\tu & \in L^1((0,\rT);H^1(\Omega)) \cap L^\infty ((0,\rT); L^2(\Omega))\cap L^k ((0,\rT); L^\infty(\Omega)) \cap L^{\frac{k}{k-1}}((0,\rT);L^s(\Omega)),   \label{assumption}  \\
 R_1 & \in L^2\left((0,\rT)\times\Omega\right),\\ 
 R_2  &\in H^1\left((0,\rT)\times\Omega)\right)\cap L^1\left((0,\rT);L^\infty(\Omega)\right) \cap L^{\frac{k}{k-1}}(0,\rT);L^2(\Omega)). 
\end{align*}
We define the  energy $
 \cI[u(t,\cdot)]:= \frac{1}{2} \int_\Omega |u(t,\cdot)|^2 + |\partial_t u(t,\cdot)|^2 + |\nabla u(t,\cdot)|^2 dx.$ 
Then we have the inequality
\begin{align}\begin{split}
 \sup_{0 \leq s \leq t} \cI[u(s,\cdot)-\tu(s,\cdot)]\leq \gamma(t)\cdot \exp\left(\int_0^t \delta(s) ds\right)\end{split}\label{inequmfd},\ \ \ \text{ where}
 \end{align}
 \vspace{-1cm}
 \begin{align*}
  \gamma(t)&:= 8\cI[u-\tu]\big|_{t=0} + 4 \|R_2(0,\cdot)\|_{L^2(\Omega)}^2 + 2 \|R_2(t,\cdot)\|_{L^2(\Omega)}^2 + 
   4 \int_0^t  \|R_1 R_2\|_{L^1(\Omega)}ds  \\
   & \quad + 8\left(  \int_0^t \|\nabla R_2\|_{L^2(\Omega)} + 2c_1 \|D\tu\|_{L^\infty(\Omega)}\|R_2 \|_{L^2(\Omega)}\right. \\
   &\quad\qquad\qquad\qquad\qquad\qquad\qquad \left.+ c_1\,\sqrt{c_2} \|R_2 \|_{L^p(\Omega)} \|D\tu\|_{L^{s}(\Omega)}^2 + \|R_1\|_{L^2(\Omega)} ds\right)^2,\\
  \delta(t)& : = 4+ 4 c_1 \Big(3c_2 \|D\tu\|^2_{L^{2s}(\Omega)}  +  2 c_3\|\pt\tu\|_{L^\infty(\Omega)}   + 2  \|R_2 \|_{L^\infty(\Omega)} \Big).
 \end{align*}
%
The constant $c_3$ is the diameter of $N$, $c_3=\operatorname{diam}(N)$, $c_2$ is the squared constant from the Sobolev embedding $H^1 \hookrightarrow L^{p}$, $p:=\frac{2s}{s-2}$. The constant $c_1= c_1(N)$ is the maximum of the Lipschitz constant of $p\mapsto A[p](\cdot,\cdot)$ and $\sup_N |A|$.
\end{theorem}

\begin{proof}
We pretend that $u$ is a strong solution in order to keep the computations simpler. For a weak solution, we use the cutoff-trick with $\phi_\epsilon$, that was introduced in the proof of Theorem~\ref{thm:stab}. 
 Let us denote 
 \begin{align*}
\frac{d}{dt} \cI[u-\tilde u] = \frac{d}{dt}\left(\tfrac{1}{2}\hspace{-0.1cm}\int_\Omega |u-\tilde u|^2 + |\partial_t u -\partial_t\tu|^2 + |\nabla u-\nabla\tu|^2 dx\right)=: \frac{d}{dt} (I^1 + I^2 + I^3).
 \end{align*}
 Then we compute 
 \begin{align*}
  \frac{d}{dt} (I^2 + I^3) &= (\partial_{t}u-\partial_t\tilde u,\partial_{tt}u-\ptt\tu)_\Omega+ (\nabla u-\nabla\tu,\pt\nabla u-\pt\nabla\tu)_\Omega\\
  &= (\pt u-\pt\tu,\ptt u-\ptt\tu)_\Omega - (\Delta u-\Delta\tu,\pt u-\pt \tu)_\Omega\\
  &= (\pt u - \pt\tu, A[u](Du,Du)- A[\tu](D\tu,D\tu)-R_1 - \pt R_2)_\Omega.
 \end{align*}
As $A$ has normal values, we get that $\pt u \cdot A[u](\cdot,\cdot)=0= \pt \tu \cdot A[\tu](\cdot,\cdot)$. Using this orthogonality we get that
\begin{align}\begin{split}
  \frac{d}{dt} (I^2 + I^3) &= \left(\pt u, A[u](D\tu,D\tu)-A[\tu](D\tu,D\tu)\right)_\Omega - \left(\pt u- \pt\tu, R_1 + \pt R_2\right)_\Omega\\
  &\qquad\qquad - (\pt \tu, A[u](Du,Du)-A[\tu](Du,Du))_\Omega\\
  &= (\pt u-\pt\tu, (A[u]-A[\tu])(D\tu,D\tu))_\Omega  - (\pt u- \pt\tu, R_1 + \pt R_2)_\Omega\\
  &\qquad\qquad- (\pt \tu, (A[u]-A[\tu])(Du,Du)- (A[u]-A[\tu])(D\tu,D\tu))_\Omega.\end{split} \label{eq1}
\end{align}
 We consider the term $(\pt u-\pt\tu, \pt R_2)_\Omega$ that is contained in (\ref{eq1}) and compute
\begin{align}\begin{split}
 (\pt u-\pt\tu, \pt R_2)_\Omega & = \frac{d}{dt}(\pt u-\pt\tu,  R_2)_\Omega  - (\Delta u - \Delta\tu, R_2)_\Omega\\
 & \ \ -( A[u](Du,Du) - A[\tu](D\tu,D\tu)- R_1 - \pt R_2, R_2)_\Omega\\
 & = \frac{d}{dt}(\pt u-\pt\tu,  R_2)_\Omega + (\nabla u - \nabla \tu,\nabla R_2)_\Omega \\
 & \ \ - (A[u](Du,Du) - A[\tu](D\tu,D\tu), R_2)_\Omega + (R_1, R_2)_\Omega \\
 &\ \ + \frac{1}{2}\frac{d}{dt} ( R_2 , R_2 )_\Omega. \end{split}\label{eq2}
\end{align}
In order to estimate the term $(A[u](Du,Du) - A[\tu](D\tu,D\tu), R_2 )_\Omega$ we do the following:
\begin{align*}
 &\big( A[u](Du,Du) - A[\tu](D\tu,D\tu)\big)\cdot R_2 \\
 &= \big(A[u](Du,Du) - A[u](D\tu,D\tu) + (A[u]-A[\tu])(D\tu,D\tu)\big) \cdot R_2 \\
  &= \big(A[u](Du-D\tu,Du-D\tu) + 2A[u](Du-D\tu,D\tu) + (A[u]-A[\tu])(D\tu,D\tu)\big) \cdot R_2 
\end{align*}
where we used the equation $B(X,X)- B(Y,Y) = B\left(X-Y, (X-Y) + 2Y\right)$ for a symmetric bilinear form $B:\rR^n\times\rR^n\to \rR$. The map $p\mapsto A[p](\cdot,\cdot)$ is Lipschitz and $N$ is a compact smooth manifold. Therefore, there is a constant $c_1 = c_1(N)$ such that
\begin{align*}
 |\big(A[u](Du,Du) &- A[\tu](D\tu,D\tu)\big)\cdot R_2 |\\
 &\leq c_1 |R_2 |\left( |Du-D\tu|^2 + 2 |Du-D\tu||D\tu| + |u-\tu||D\tu|^2 \right)
\end{align*}
Using H\"older's  inequality and generalized H\"older's inequality we get that
\begin{align}\begin{split}
             & ((A[u](Du,Du) - A[\tu](D\tu,D\tu), R_2 )_\Omega\\
              &\qquad \leq 2 c_1 \|R_2 \|_{L^\infty(\Omega)} \cI[u-\tu] +  2 c_1 \|D\tu\|_{L^\infty(\Omega)}\|R_2 \|_{L^2(\Omega)}\sqrt{ 2\cI[u - \tu]} \\
              & \ \ \ \qquad + c_1 \|R_2 \|_{L^p(\Omega)} \|u-\tu\|_{L^p(\Omega)} \|D\tu\|_{L^{2q}(\Omega)}^2
             \end{split}\label{eq3}
\end{align}
for $\frac{2}{p} + \frac{1}{q}=1$. For $m\geq 3$, we use this inequality for $p=\frac{2m}{m-2}$ and get by Sobolev embedding $H^1(\Omega) \hookrightarrow L^{\frac{2m}{m-2}}(\Omega)$
\begin{align}
\begin{split} 
 & ((A[u](Du,Du) - A[\tu](D\tu,D\tu), R_2 )_\Omega\\
              &\qquad \leq 2 c_1 \|R_2 \|_{L^\infty(\Omega)} \cI[u-\tu] +  2 c_1 \|D\tu\|_{L^\infty(\Omega)}\|R_2 \|_{L^2(\Omega)}\sqrt{ 2\cI[u - \tu]} \\
              & \ \ \ \qquad + c_1 \, \sqrt{c_2} \|R_2 \|_{L^p(\Omega)} \|D\tu\|_{L^{2q}(\Omega)}^2 \sqrt{ 2\cI[u - \tu]},
\end{split}\label{eq3b}
\end{align}
where $q = \frac{m}{2}$ and $c_2$ is the squared constant from the Sobolev embedding. Thus, in the case $m\geq 3$, we get for $s=m$ the inequality
\begin{align}
\begin{split} 
 & ((A[u](Du,Du) - A[\tu](D\tu,D\tu), R_2 )_\Omega\\
              &\qquad \leq 2 c_1 \|R_2 \|_{L^\infty(\Omega)} \cI[u-\tu] \\
              &\qquad \ \ \ +  c_1\left( 2\|D\tu\|_{L^\infty(\Omega)}\|R_2 \|_{L^2(\Omega)} + \sqrt{c_2} \|R_2 \|_{L^p(\Omega)} \|D\tu\|_{L^{s}(\Omega)}^2\right)\sqrt{ 2\cI[u - \tu]}.
\end{split}\label{eq3c}
\end{align}
For $m=2$, we can choose $1\leq p<\infty$ arbitrarily large in (\ref{eq3}), which implies (\ref{eq3c}) holds for $2<s$ arbitrary, in particular for $s$ close to two. 
\\
We put (\ref{eq2}) and (\ref{eq3c}) together and get that
\begin{align}\begin{split}
              &-(\pt u-\pt\tu , \pt R_2 )_\Omega  \leq 
              -\frac{d}{dt}(\pt u-\pt\tu,  R_2)_\Omega + 2 c_1 \|R_2 \|_{L^\infty(\Omega)} \cI[u-\tu] \\
 & \ \  + \left(\|\nabla R_2 \|_{L^2(\Omega)}+ 2c_1 \|D\tu\|_{L^\infty(\Omega)}\|R_2 \|_{L^2(\Omega)} + c_1\,\sqrt{c_2} \|R_2 \|_{L^p(\Omega)} \|D\tu\|_{L^{s}(\Omega)}^2\right)\sqrt{2 \cI[u - \tu]}  \\
 &\ \ + \|R_1 R_2\|_{L^1(\Omega)}- \frac{1}{2}\frac{d}{dt} ( R_2 , R_2 )_\Omega
             \end{split}\label{eq4}
\end{align}
Coming back to (\ref{eq1}), we again use the formula \\
$B(X,X)- B(Y,Y) = B\left(X-Y, (X-Y) + 2Y\right)$ for the terms with the second fundamental form there. Choosing any $\frac{1}{s} + \frac{1}{\tilde s}= \frac{1}{2}$ yields
\begin{align*}
(\pt u-\pt\tu, &  (A[u]-A[\tu])(D\tu,D\tu))_\Omega  - (\pt \tu, (A[u]-A[\tu])(Du,Du)\\
&\qquad\qquad\qquad\qquad\qquad   - (A[u]-A[\tu])(D\tu,D\tu))_\Omega\\
  &\leq c_1\|\pt u-\pt\tu\|_{L^2(\Omega)}\|u-\tu\|_{L^{\tilde s}(\Omega)}\|D\tu\|^2_{L^{2s}(\Omega)}\\
 &\qquad  + c_1 \|\pt\tu\|_{L^\infty(\Omega)} \|u-\tu\|_{L^\infty(\Omega)}\|Du-D\tu\|^2_{L^2(\Omega)} \\
 &\qquad  + 2c_1 \|(|\pt\tu||D\tu|)\|_{L^s(\Omega)} \|u-\tu\|_{L^{\tilde s}(\Omega)}\|Du-D\tu\|_{L^2(\Omega)}.
\end{align*}
We define $c_3 = c_3(N) = \operatorname{diam}(N)$ and $c_2$ to be the squared Sobolev embedding constant from $H^1 \hookrightarrow L^{\tilde s}$. We also use $|\pt \tu||D\tu|\leq |D\tu|^2$ and get that
\begin{align}\begin{split}
 \frac{d}{dt}& (I^2 + I^3) \leq c_1 \Big( c_2\|D\tu\|^2_{L^{2s}(\Omega)}  + 2c_3 \|\pt\tu\|_{L^\infty(\Omega)} \\
 &\qquad\qquad\qquad+ 2 c_2\|D\tu\|_{L^{2s}(\Omega)}^2  + 2  \|R_2 \|_{L^\infty(\Omega)} \Big) \cI[u-\tu]\\
 &\qquad\qquad \qquad-   \frac{d}{dt}(\pt u-\pt\tu,  R_2)_\Omega  + \|R_1 R_2\|_{L^1(\Omega)}  - \frac{1}{2}\frac{d}{dt} ( R_2 , R_2 )_\Omega
 \\ &\qquad\qquad \qquad + \left(\|\nabla R_2 \|_{L^2(\Omega)}+ 2c_1 \|D\tu\|_{L^\infty(\Omega)}\|R_2 \|_{L^2(\Omega)} \right.\\
 &\qquad\qquad\qquad \left.+ c_1\,\sqrt{c_2} \|R_2 \|_{L^p(\Omega)} \|D\tu\|_{L^{s}(\Omega)}^2 + \|R_1\|_{L^2(\Omega)}\right)\sqrt{2\cI[u-\tu]}.
 \end{split} \label{inequfast}
\end{align}
Together with $\frac{d}{dt}I^1 = (u-\tu,\pt u-\pt \tu)_\Omega \leq \cI[u-\tu]$, (\ref{inequfast}) implies
\begin{align}\begin{split}
 \cI[&u(t,\cdot)-\tu(t,\cdot)] \leq \cI[u-\tu]\big|_{t=0} -(\pt u-\pt\tu, R_2 )_\Omega\big|_t + (\pt u(0)-\pt\tu (0),R_2(0,\cdot))_\Omega \\
 &\qquad -\frac{1}{2}\|R_2(t,\cdot)\|_{L^2(\Omega)}^2  +\frac{1}{2}\|R_2(0,\cdot)\|_{L^2(\Omega)}^2\\
 &  \qquad+ \int_0^t \alpha(s) \cI[u(s,\cdot)-\tu(s,\cdot)] ds + \int_0^t\beta(s)ds  + \int_0^t\gamma(s)\sqrt{2 \cI[u(s,\cdot)-\tu(s,\cdot)]}ds  ,\end{split}\label{eq5}
\end{align}
where
\begin{align*}
 \alpha(t)& : = 1+ c_1 \Big(3c_2 \|D\tu\|^2_{L^{2s}(\Omega)}  + 2c_3 \|\pt\tu\|_{L^\infty(\Omega)}   + 2  \|R_2 \|_{L^\infty(\Omega)} \Big) \\
 \beta(t)&:= \|R_1 R_2\|_{L^1(\Omega)}\\
 \gamma(t)&:= \|\nabla R_2 \|_{L^2(\Omega)}+ 2c_1 \|D\tu\|_{L^\infty(\Omega)}\|R_2 \|_{L^2(\Omega)} + c_1\, \sqrt{c_2} \|R_2 \|_{L^p(\Omega)} \|D\tu\|_{L^{s}(\Omega)}^2 +  \|R_1\|_{L^2(\Omega)}.
\end{align*}
From inequality (\ref{eq5}), we get that
\begin{align}\begin{split}
 \frac{1}{2}& \cI[u(t,\cdot)-\tu(t,\cdot)] \leq 2\, \cI[u-\tu]\big|_{t=0}   
 +\|R_2(0,\cdot)\|_{L^2(\Omega)}^2+ \frac{1}{2}\|R_2(t,\cdot)\|_{L^2(\Omega)}^2\\
 &  + \int_0^t \alpha(s) \cI[u(s,\cdot)-\tu(s,\cdot)] ds+ \int_0^t\beta(s)ds
 + \int_0^t\gamma(s)\sqrt{2\cI[u(s,\cdot)-\tu(s,\cdot)]}ds.
 \end{split}\label{eq6}
\end{align}
We set 
\[ \bar \cI[u-\tu](t):= \sup_{0 \leq s \leq t} \cI[u(t,\cdot)-\tu(t,\cdot)]\]
and obtain, as the right hand side of \eqref{eq6} is monotone increasing in time
\begin{align*}
 \frac{1}{2}& \bar \cI[u(t,\cdot)-\tu(t,\cdot)] \leq  2\, \cI[u-\tu]\big|_{t=0}   
 +\|R_2(0,\cdot)\|_{L^2(\Omega)}^2 +\frac{1}{2}\|R_2(t,\cdot)\|_{L^2(\Omega)}^2\\
 &  \qquad+ \int_0^t \alpha(s) \bar \cI[u(s,\cdot)-\tu(s,\cdot)] ds+ \int_0^t\beta(s)ds
 + \int_0^t\gamma(s)ds \sqrt{2\bar \cI[u(t,\cdot)-\tu(t,\cdot)]}.
\end{align*}so that
\begin{align*}
 \frac{1}{4} \bar \cI[u(t,\cdot)-\tu(t,\cdot)] &\leq  2\, \cI[u-\tu]\big|_{t=0}   
 +\|R_2(0,\cdot)\|_{L^2(\Omega)}^2+ \frac{1}{2}\|R_2(t,\cdot)\|_{L^2(\Omega)}^2\\
 &  + \int_0^t \alpha(s) \bar \cI[u(s,\cdot)-\tu(s,\cdot)] ds+ \int_0^t\beta(s)ds
 + 2\left(\int_0^t\gamma(s)ds\right)^2.
\end{align*}
Gronwall's inequality now implies (\ref{inequmfd}).\\[-0.2cm]

It remains to explain that the conditions on $\tu$ are strong enough for the above computations. We use $\partial_t\tu$ as a test function in the weak formulation for $u$. This is allowed because the assumptions on $\tu$ imply $\partial_t \tu\in W^{1,1}((0,\rT);L^2(\Omega))\cap L^1((0,\rT);L^\infty\cap H^1(\Omega)) $.\\
The conditions on $\tu$ imply that all appearing terms are finite. Note particularly that 
\begin{align*}
 &\int \|R_2\|_{L^2(\Omega)} \|D\tu\|_{L^\infty(\Omega)} dt \leq  c \int \|R_2\|_{L^2(\Omega)}^{\frac{k}{k-1}}  + \|D\tu\|_{L^\infty(\Omega)}^k dt <\infty \ \ \ \text{ and}\\
 &\int \|D\tu\|_{L^{2s}(\Omega)}^2 dt\leq \int \|D\tu\|_{L^{\infty}(\Omega)}\|D\tu\|_{L^{s}(\Omega)} dt \leq \int c\|D\tu\|_{L^{\infty}(\Omega)}^k + \|D\tu\|_{L^{s}(\Omega)}^{\frac{k}{k-1}} dt <\infty,\\
 & \int  \|R_2 \|_{L^p(\Omega)} \|D\tu\|_{L^s}^2 dt \leq c \int  \|R_2\|_{H^1}^2  + \|D\tu\|_{L^s}^4 dt,
\end{align*}
and the latter is finite because $k\leq \frac{4}{3}$ implies $\frac{k}{k-1} \geq 4$.
\end{proof}

 \begin{remark}\label{rem:stabdiff}
  As one can see, the estimate for general $N$ is slightly different than the one for spheres. The reason is that there is no such formulation as the \emph{angular momentum formulation} (\ref{wmom}) for general $N$. The formulation (\ref{wmom}) has the advantage that it transforms the wave map equation into a system of equations that are only first order in $t$. Our numerical scheme is based on that formulation. Our reconstruction will not be $W^{2,1}$ but only  $W^{1,\infty}$, a time derivative is a weak derivative. Furthermore, in the formulation (\ref{wmom}) the term $u\times\omega$ is the time derivative $\pt u$, which makes sense for $u$ and $\omega$ being only continuous. This is also the reason that the two energies $\cH$ and $\cI$ differ in the time derivative.

  Our motivation for splitting the residual in Theorem \ref{thm:stabN} is that if $\tu$ in Theorem \ref{thm:stab} happens to be twice weakly differentiable in time, then it satisfies
  \[
   \partial_t^2 \tu = \Delta \tu + ( |\nabla \tu|^2 - |\tu \times \tw|^2) \tu + \partial_t r_u + r_g   + r_u \times \tw 
   + \tu \times r_w,
  \]
  where we we have used \eqref{2801b}.
  This equation provides a connection between $R_1$, $R_2$  and $r_u$, $r_g$ and $r_\omega$.
  We formulated the estimates in Theorem~\ref{thm:stab} and Theorem~\ref{thm:stabN} as similar as possible -- in particular there is no norm of a time derivative of the residuals involved on the right hand sides of the inequalities -- but in detail the estimates look a little bit different.
  \end{remark}
  
  \begin{remark}\label{rem:stabdiff1B}
  We note that our regularity assumptions for general $N$ are only silghtly stronger than the one for spheres in Theorem~\ref{thm:stab}. On the one hand, we need a second time derivative of $\tu$, see the remark above. On the other hand, we need the additional assumption $D\tu\in L^{\frac{k}{k-1}}((0,\rT);L^s(\Omega))$ where $\frac{k}{k-1}\geq 4$ (and not $k \leq 2$ as in the sphere case). The reason is the appearance of the term $\int_0^t \| D\tu\|^2_{L^s(\Omega)} \|R_2\|_{H^1(\Omega)}ds$ in inequality (\ref{inequmfd}). Due to the different structure of the proof, this term did not appear in Theorem~\ref{thm:stab}. \\[-0.2cm]

  Note furthermore that by choosing $k=1$ and considering the case $R_1=R_2=0$ we have the same regularity assumptions as Struwe has in his result \cite[Thm~2.2]{Struwe_1999}, see Remark~\ref{rem:regs}.
 \end{remark}

 \begin{remark}\label{rem:stabdiff2}
  In the next chapter, we derive an a posteriori error estimate based on the stability framework of Theorem \ref{thm:stab} and not on Theorem \ref{thm:stabN}. 
 Since an \textit{angular momentum formulation} is not available for general target manifolds,
  basing our analysis on Theorem \ref{thm:stab}, requires us to restrict our analysis to numerical methods for wave maps with values in $S^2$.
  The reason why we use this, less general, stability analysis is its ability to handle reconstructions $\tu, \tw$ that are in $W^{1, \infty}$ in time. In contrast, a posteriori results based on \ref{thm:stabN} would require some reconstruction of the numerical solution that is in $W^{2,1}$ in time. Such a reconstruction was derived for the linear wave equation in \cite{Georgoulis_Lakkis_Makridakis_2013} but it is not clear how to extend this construction to the wave map case.
 \end{remark}

 \section{Numerical scheme and a posteriori error estimates}
 We are concerned with a semi-discretization in time devised for wave maps into spheres in \cite{Karper_Weber_2014}.
 The scheme is based on the reformulation \eqref{wmom} of the wave-map equation 
 and reads
 \begin{equation}\label{scheme} 
d_t u^{k+1} = u^{k+1/2} \times \omega^{k+1/2} , \quad 
d_t \omega^{k+1} = \Delta u^{k+1/2} \times u^{k+1/2},
 \end{equation}
 where $d_t$ is a backward difference quotient in time for a step size $\tau_k$ and the fractional superscript denotes an average in time, i.e.
 \[ d_t u^{k+1} := \frac 1 {\tau_k} (u^{k+1} - u^k) , \quad 
  u^{k+1/2}:= \frac{u^{k+1} + u^k}{2}.
 \]
 Note that \eqref{scheme} preserves the point-wise constraints $|u^k(x)|=1$ and $u^k \perp \omega^k$ for all $k$ provided they are satisfied for the initial data. Indeed, this follows by multiplying \eqref{scheme} by $u^{k+1/2}$.
 
A fully discrete version of this scheme, using a finite element discretization in space, was investigated in
 \cite{Bartels_2015}. There, it was shown that the scheme conserves energy, in the sense that
 \[
  \frac  12  \int_\Omega |\omega^k|^2 + |\nabla u^k|^2 dx = \frac  12  \int_\Omega |\omega^0|^2 + |\nabla u^0|^2 dx \qquad \forall k=0,\dots, N
 \]
and that in the limit of vanishing  time step size subsequences of the numerical solution converge to weak solutions of the wave map problem.
 In contrast to the results of Bartels, we are interested in a posteriori error estimates, i.e.,  computable error bounds that can be evaluated once the numerical solution has been computed. 
 
 While weak solutions can be defined beyond times at which a gradient blow-up has occurred there does not seem to be uniqueness beyond these times and,  the results of \cite{DAncona_Georgiev_2004} show that numerical schemes cannot be expected to converge with respect to the energy norm, once the exact solution has no additional regularity. As a consequence, we are only able to provide useful estimates up to the blow up time.
 Beyond apparent gradient blow-up in the solution the error bounds keep converging in $\tau$ but blow up 
 for $h \rightarrow 0$.

 Our error analysis follows the approach outlined by Makridakis \cite{Makridakis_2007} in that it combines the 'energy type' stability result derived in Theorem \ref{thm:stab} with a suitable reconstruction of the numerical solution.

 The scheme \eqref{scheme} is very close to a Crank-Nicolson scheme and consequently we apply a reconstruction that is close to the reconstruction proposed in \cite{Akrivis_Makridakis_Nochetto_2006}.
 A specific feature is that in order to employ the stability result from Theorem \ref{thm:stab} we need to project the reconstruction into the target manifold. 
 
\subsection{Reconstruction}
In the sequel, we will define suitable reconstructions of the numerical solutions assuming that a sequence of numerical approximations at different points in time  $0=t_0 < t_1 < .... < t_N$ is given:
\[ \{ u^n\}_{n=0}^N : \Omega \rightarrow  S^2, \{ \omega^n\}_{n=0}^N : \Omega \rightarrow \rR^3.\]
Firstly, we define preliminary, globally continuous, and piecewise linear interpolants
using local Lagrange polynomials
\[ \ell_n^0(t):=\frac{t_{n+1}-t}{t_{n+1}-t_n}, \qquad \ell_n^1(t): = \frac{t-t_n}{t_{n+1}-t_n}\]
by
\begin{equation*}
 \widehat u |_{[t_n, t_{n+1}]}(t) := \ell_n^0(t)u^n + \ell_n^1(t)u^{n+1}, \qquad
  \widehat \omega |_{[t_n, t_{n+1}]}(t): = \ell_n^0(t) \omega^n +\ell_n^1(t)\omega^{n+1} .
\end{equation*}
%
In addition, for any $g\in C^0 ([0,t_N], L^2 (\Omega, \rR^3))$ we define  piecewise constant and piecewise linear interpolants by
\begin{align*}
I_1 [g] |_{[t_n, t_{n+1}]}(t) & := g(t_n) + \frac{t-t_n}{t_{n+1}-t_n} (g(t_{n+1}) - g(t_n))\\
I_0[g]|_{(t_n, t_{n+1})}(t) & := \frac 12 ( g(t_n) + g(t_{n+1}) ).
\end{align*}
%
This allows us to rewrite the  numerical scheme as
\begin{equation}\label{eq:nama}
 \begin{aligned}
  \partial_t \widehat u|_{(t_n, t_{n+1})}  = I_0[ \widehat u \times \widehat \omega]  - \frac{u^{n+1} - u^n}{2} \times \frac{\omega^{n+1} - \omega^n}{2}  =:  I_0[ \widehat u \times \widehat \omega]  - a_u^n\,, \\
   \partial_t \widehat \omega|_{(t_n, t_{n+1})}  = I_0[\Delta \widehat u \times \widehat u] - \frac{\Delta u^{n+1} - \Delta  u^n}{2} \times \frac{u^{n+1} - u^n}{2} =:  I_0[ \Delta \widehat u \times \widehat u]  - a_\omega^n \,.
 \end{aligned}
\end{equation}
Next, we define piecewise quadratic reconstructions via
\begin{equation}\label{recon2}
  \begin{aligned}
   u^*|_{(t_n, t_{n+1})}(t)  & := u^n + \int_{t_n}^{t} I_1[ \widehat u \times \widehat \omega]    - a_u^n\, ds, \\
   \tw|_{(t_n, t_{n+1})} (t) & := \omega^n  + \int_{t_n}^{t} I_1[ \Delta \widehat u \times \widehat u]  - a_\omega^n\, ds.
 \end{aligned}
\end{equation}
Note that $u^*, \tw$ are globally continuous in time since the trapezoidal formula is exact for linear functions and, in particular, $u^*(t_n)= u^n$ and, thus, $|u^*(t_n,x)|= |u^n(x)|=1$ for all $n$ and all $x \in \Omega$.
Finally, we define 
\begin{equation}\label{recon3}
  \tu := \frac{u^*}{|u^*|},
\end{equation}
such that $\tu$ is a map into $S^2$ and $\tu(t_n)=u^n$ for $n=0,\dots,N$.\\
A straightforward computation gives
\begin{equation}\label{def:res}
 \begin{aligned}
  \partial_t \tu&= \tu \times \tw + r_u \quad \text{with } r_u := \tu \times \tw - I_1 (\tu \times \tw) + a_u + \partial_t u^* - \frac{\partial_t u^*}{|u^*|} + \frac{ \partial_t u^* \cdot u^*}{|u^*|^3}u^*,\\
  \partial_t \tw& = \Delta \tu \times \tu + r_\omega \quad \text{with } r_\omega := \Delta \tu \times \tu - I_1 [\Delta \tu \times \tu] + a_\omega,
 \end{aligned}
\end{equation}
where $a_u|_{(t_n,t_{n+1})}:=a_u^n$ and $a_\omega|_{(t_n,t_{n+1})}:=a_\omega^n$.

\subsection{Computable bounds for residuals}
 It should be noted that, due to the projection onto the sphere, $\tu$ is no longer piecewise quadratic.
 Thus, it is not straightforward how to compute (norms of) the residuals from \eqref{def:res}. In addition, the method at hand is formally  second order, so that we should strive for a reconstruction making the a posteriori error estimator second order as well.
 It might seem obvious that $\tu \times \tw - I_1 (\tu \times \tw) + a_u $ and $\Delta \tu \times \tu - I_1 [\Delta \tu \times \tu] + a_\omega$ are second order in time,
 but this property does not seem obvious for $\partial_t u^* - \frac{\partial_t u^*}{|u^*|} + \frac{ \partial_t u^* \cdot u^*}{|u^*|^3}u^*$.
 
 Let us begin by decomposing the residuals into several parts:
 \begin{equation}
  \begin{aligned}
   r_{u,1}&=  \tu \times \tw - I_1 (\tu \times \tw), \\
   r_{u,2}&=  a_u, \\
   r_{u,3}&=  \partial_t u^* - \frac{\partial_t u^*}{|u^*|} + \frac{ \partial_t u^* \cdot u^*}{|u^*|^3}u^*,\\
   r_{\omega,1}&=\Delta \tu \times \tu - I_1 [\Delta \tu \times \tu] ,\\
   r_{\omega,2}&= a_\omega.
  \end{aligned}
 \end{equation}
Let us note that $r_{u,2},r_{\omega,2}$ are both computable from the numerical solution, without computing any reconstruction, and are both going to converge to zero as $\tau^2$ as long as the scheme is at least second order convergent.

We will also give easily computable (though lengthy) bounds for the other parts of the residuals but we  postpone their proof to the appendix.
We will state the bounds for a representative time interval $(t_n, t_{n+1})$ and use the following abbreviations:
\begin{equation*}
 \begin{aligned}
  A^u &:= |u^{n+1} - u^n| , \quad A^u_x :=  |\nabla (u^{n+1} - u^n)|, \quad  A^u_{xx} :=  |\Delta (u^{n+1} - u^n)|,\\
  A^\omega &:= |w^{n+1} - w^n| , \quad A^\omega_x :=  |\nabla (w^{n+1} - w^n)| \\
   B^u &:= |u^{n+1}\times w^{n+1} - u^n\times w^n| , \quad B^u_x :=  |\nabla (u^{n+1}\times w^{n+1} - u^n\times w^n)|, \\
   B^u_{xx} &:=  |\Delta (u^{n+1}\times w^{n+1} - u^n\times w^n)|\\
  B^\omega &:= | \Delta u^{n+1}\times u^{n+1} - \Delta u^n\times u^n| , \quad B^\omega_x := |\nabla( \Delta u^{n+1}\times u^{n+1} - \Delta \tu^n\times \tu^n)|, \\
  C^\omega &:= \max\{ |\omega^n|, |\omega^{n+1}|\} ,  C^u_x := \max\{ |\nabla u^n|, |\nabla u^{n+1}|\} ,\\
  C^\omega_x &:= \max\{ |\nabla \omega^n|, |\nabla \omega^{n+1}|\} , \quad  C^u_{xx}:= \max\{ |\Delta u^n|, |\Delta u^{n+1}|\}.
 \end{aligned}
\end{equation*}
Let us mention that we expect $A^u,A^\omega, A^u_x,A^\omega_x, A^u_{xx}, B^u,B^\omega, B^u_x,B^\omega_x, B^u_{xx}$ to scale like $\tau_n:=t_{n+1} - t_n$ as long as the exact solution is regular enough for the true error to be proportional to $\tau_n^2$.
In addition, we expect $C^\omega, C^u_x,C^\omega_x, C^u_{xx}$ to be bounded (uniformly in $\tau_n$) as long as the exact solution is regular. Both expectations are confirmed by our numerical experiments.

\begin{lemma}\label{lem:compbounds}
Let us denote $\tau_n := t_{n+1} - t_n$ and let the time-step be chosen sufficiently small such that $(A^u)^2 + \tau_n B^u < \frac 14$ holds point-wise.
Then, the residuals defined in \eqref{def:res} satisfy the following point-wise estimates.
 \begin{align}\label{bound:ru1}
 &\Big| r_{u,1}|_{(t_n,t_{n+1})}\Big| \leq \tau_n B^\omega  + C^\omega (A^u)^2  + C^\omega \tau_n B^u  + \frac 14 A^u A^\omega\\
\label{bound:nru1}
 &\Big| \nabla r_{u,1}|_{(t_n,t_{n+1})} \Big| \leq (C_x^u + \tau_n B^u_x)(\tau_n B^\omega + C^\omega (A^u)^2)  + \tau_n B^\omega_x \\
 &\qquad\qquad\qquad + C^\omega ( A^u_x A^u + \tau_n B_x^u + C^u_x \tau_n B^u + \tau_n^2 B^u B^u_x) \nonumber\\
  &\qquad\qquad\qquad+ (A^u)^2 C^\omega_x +\tau_n B^u_x C^\omega + \tau_nB^u C^\omega_x + A^u_x A^\omega + A^u A^\omega_x \nonumber\\
 \label{bound:ru2}
 &\Big| r_{u,2}|_{(t_n,t_{n+1})}\Big| \leq \frac 14 A^u A^\omega\\
 \label{bound:nru2}
 &\Big|\nabla  r_{u,2}|_{(t_n,t_{n+1})}\Big| \leq \frac 14 (A^u_x A^\omega + A^u A^\omega_x)\\
\label{bound:ru3}
& \Big| r_{u,3}|_{(t_n,t_{n+1})}\Big| \leq (C^\omega + \frac 14 A^u A^\omega)(  \frac 43 (A^u)^2 +\frac 83 \tau_n B^u) \\
& \nonumber \qquad\qquad\qquad\qquad  + 4 A^u A^\omega (2 + \tau_n B^u) +4 C^\omega  \tau_n B^u
\label{bound:nru3}
\end{align}
 \begin{align}
 &\Big|\nabla r_{u,3}|_{(t_n,t_{n+1})}\Big| \leq 8 C^\omega [ A^u A^u_x + \tau_n B^u_x + C^u_x \tau_n B^u + \tau_n^2 C^u C^u_x]\\
 &\nonumber\qquad\qquad+ (C^\omega_x + C^u_x C^\omega )[(A^u)^2 + 3\tau_n B^u]\\
 &\nonumber\qquad\qquad+C^u_x A^u A^\omega + A^u A^\omega_x + \frac 14 (A^u_x A^\omega + A^u A^\omega_x) \\
 &\nonumber\qquad\qquad + \frac 14 A^u A^\omega (C^u_x + \tau_n B^u_x) + (C^\omega_x + C^u_x C^\omega)\tau_n B^u + C^\omega \tau_n B^u_x\\
 \label{bound:ro}
  &\Big| r_{\omega}|_{(t_n,t_{n+1})}\Big| \leq  ( C^u_{xx}+ \tau_n B^u_{xx}) \left( \frac 73 (A^u)^2+ \frac {11}3 \tau_n B^u\right)+\frac 94 A^u_{xx} A^u + (A^u_x)^2\\
&\nonumber\qquad\qquad+(C^u_x + \tau_n B^u_x)2[ A^u_x A^u + \tau_n B^u_x + (1 +C^u_x) \tau_n B^u + \tau_n^2 B^u_x B^u]\\
\label{bound:rg}
 &| r_g | \leq (C^\omega + \tau_n B^\omega)) [\tau_n B^\omega + C^\omega( (A^u)^2 + \tau_n B^u) + A^u A^\omega] \\
 &\nonumber\qquad\qquad + [\tau_n B^\omega + C^\omega( (A^u)^2 + \tau_n B^u) + A^u A^\omega]^2.
\end{align}
\end{lemma}
An a posteriori estimate for the difference between the exact solution and the reconstructions $(\tu,\tw)$ is easily obtained by combining Lemma \ref{lem:compbounds} and Theorem \ref{thm:stab}.
It should be noted that the definitions of the residuals and the way they enter into the error estimate is rather similar to \cite[Theorem 3.1]{Georgoulis_Lakkis_Makridakis_Virtanen_2016}. 
However, we cannot avoid exponential-in-time growth of the error due to the non-linearity of the problem, even if the solution does not exhibit gradient blow-up.

\section{Numerical experiments}
In this section, we perform numerical experiments in order to demonstrate the
scaling behaviour of the error estimator and to investigate time stepping strategies.
We focus on the scaling (in $\tau$) of the error and of the error estimator.
In order to obtain a fully practical scheme, we  used a finite difference discretisation in space 
(together with the temporal discretisation discussed above), i.e.\ the scheme from \cite{Karper_Weber_2014}.
Whenever derivatives occur in the error estimator, they are approximated using finite differences.
We consider the problem that was also studied in \cite{Bartels_2009}. 
These are smooth initial data for which a singularity forms in the numerical experiments. We believe that this reflects formation of a singularity in the exact solution, but, strictly speaking, we cannot be sure that this is what happens since the error estimator stops converging for $h \rightarrow 0$ after a singularity seems to have formed.
The fact that the error estimator does not converge once a singularity (seems to) have formed is 
related to the fact that our error estimator is an upper bound for the difference between the numerical solution and \textit{any} finite energy weak solution in the sense of Definition 3.1 and these solutions are expected not to be unique once a singularity has formed. Note that this is not a contradiction to convergence of numerical solutions to some weak solution. In order to obtain an error estimator that converges after singularity formation one would need some technique to compare the numerical solutions to only one weak solution. It is currently unclear how to achieve this.


%

 The problem data are given by $N=S^2$, $\Omega = (-\frac 12, \frac 12)^2$ and
\begin{equation}
 u(0,x) = \left\{ \begin{array}{ccc}
                \frac{(2 a(x) x_1, 2 a (x) x_2 , a(x) - |x|)^T}{a(x)^2 + |x|^2} & \text{for } & |x| \leq 1/2 \\
                (0,0,-1)^T  & \text{for } & |x| \geq 1/2
               \end{array}
\right. \qquad \partial_t u(0,\cdot) \equiv 0
\end{equation}
with $a(x):=(1-2|x|)^4$ and homogeneous Neumann boundary conditions.

The numerical solution develops a singularity at some time  $t>.2$ to be precise. Thus, we expect the scheme to converge with order $\tau^2$ in the energy norm
\[ \max_{t \in [0,\rT]} \left( \int_\Omega | \partial_t u - \tilde u \times \tilde \omega|^2 +
|\nabla u - \nabla \tilde u|^2  dx \right)^{\frac 12}\]
at least until $\rT=.2$.
\newcommand{\figscale}{1.0}
\newcommand{\figwidth}{\textwidth}



\begin{figure}[ht!]
\caption[]
          {
            \label{fig:snapshots} 
            Snapshots of numerical solution using $h=1/60$ and $\tau=2^{-10}$ at different times. Arrows depict first two components of $u_h$, color indicates direction of third component. }
          \begin{center}
            \begin{subfigure}[b]{0.24\linewidth}
              \includegraphics[width=\linewidth]
                              {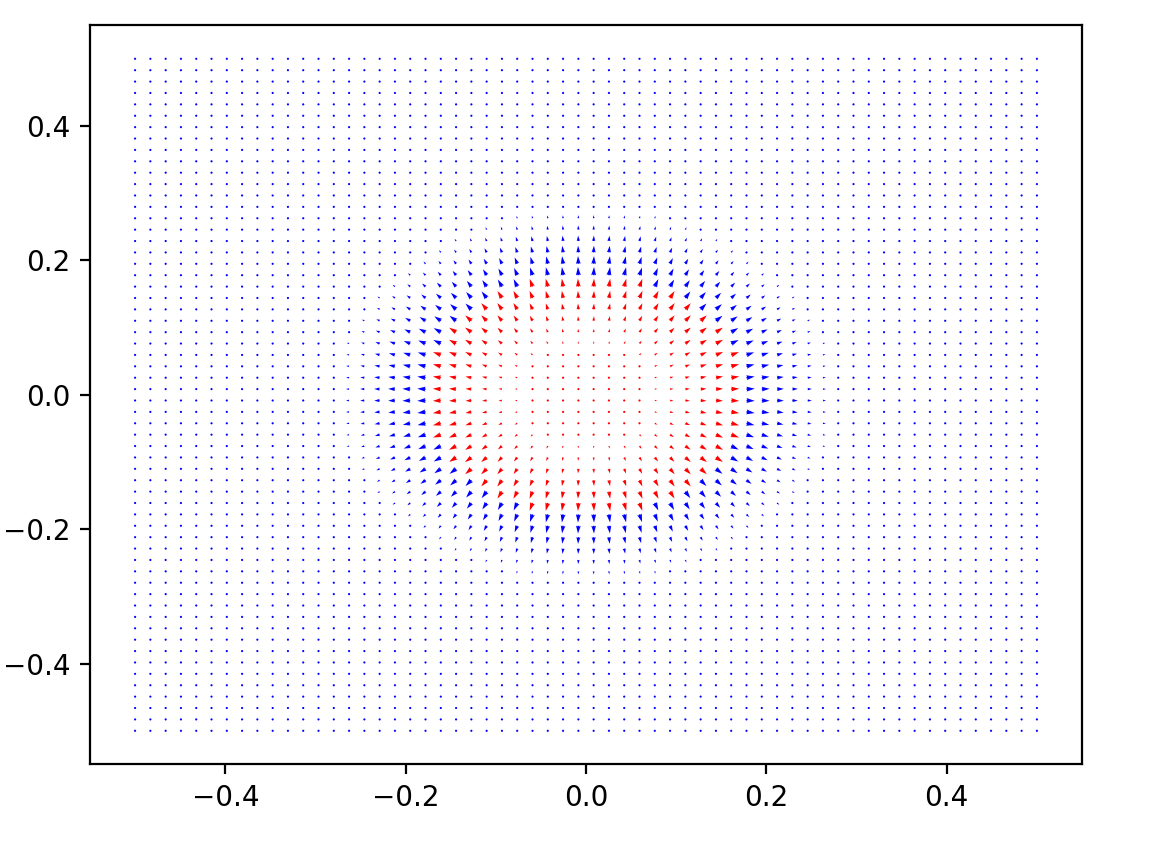}
              \caption{$t=0$}
            \end{subfigure}
            \begin{subfigure}[b]{0.24\linewidth}
              \includegraphics[width=\linewidth]
                              {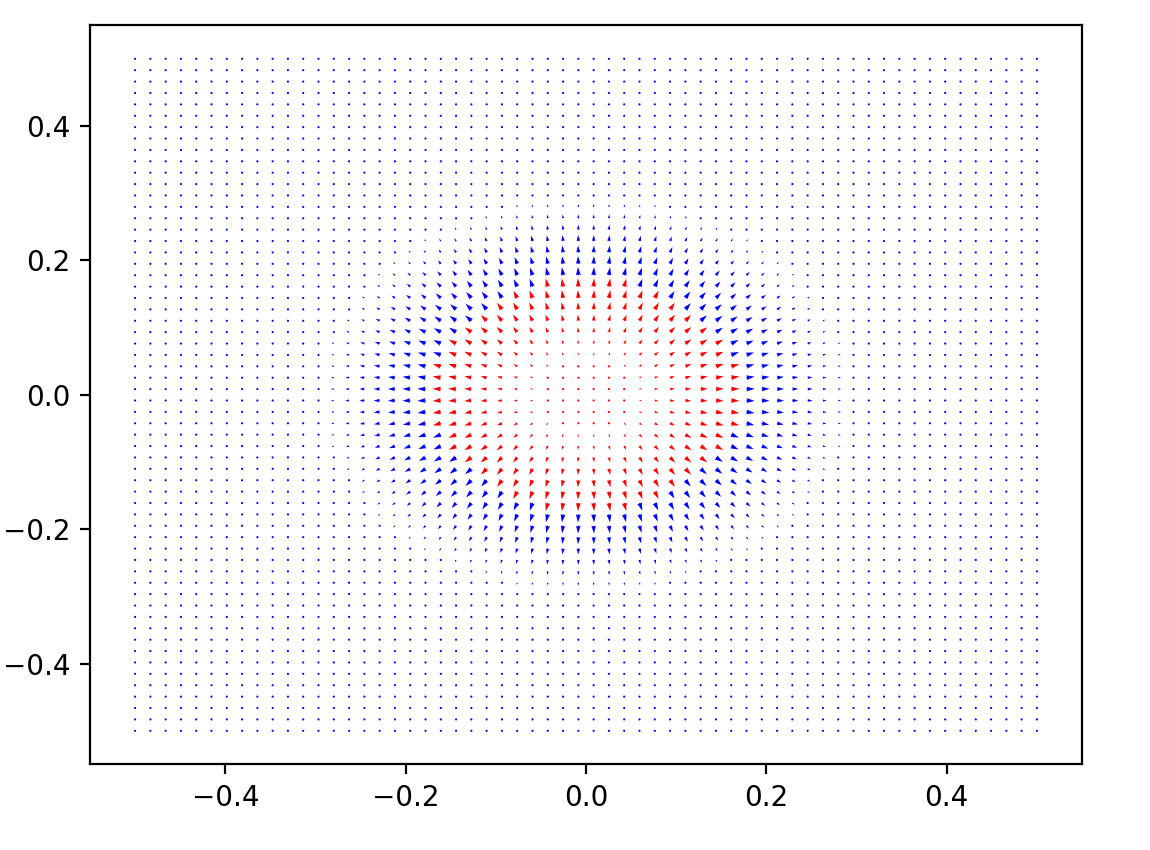}
              \caption{$t=0.049$}
            \end{subfigure}
            \begin{subfigure}[b]{0.24\linewidth}
              \includegraphics[width=\linewidth]
                              {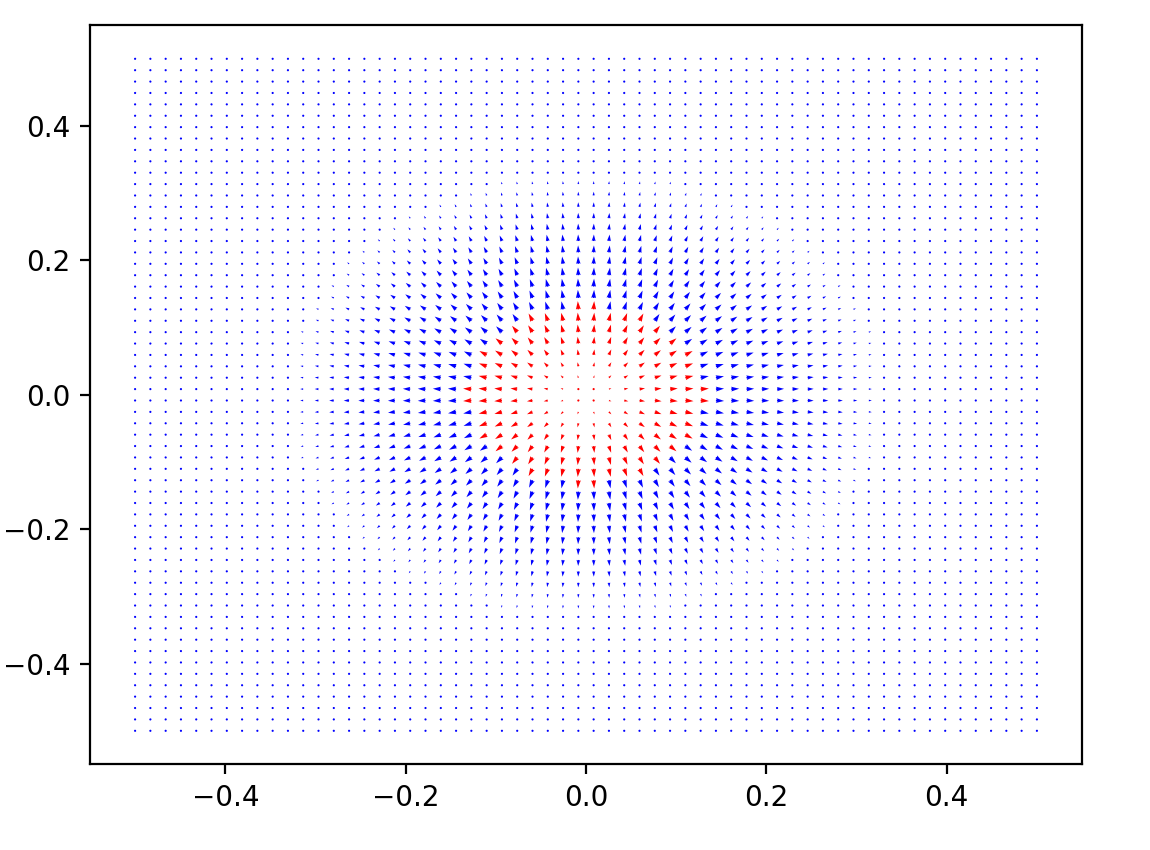}
              \caption{$t=0.098$}
            \end{subfigure}
            \begin{subfigure}[b]{0.24\linewidth}
              \includegraphics[width=\linewidth]
                              {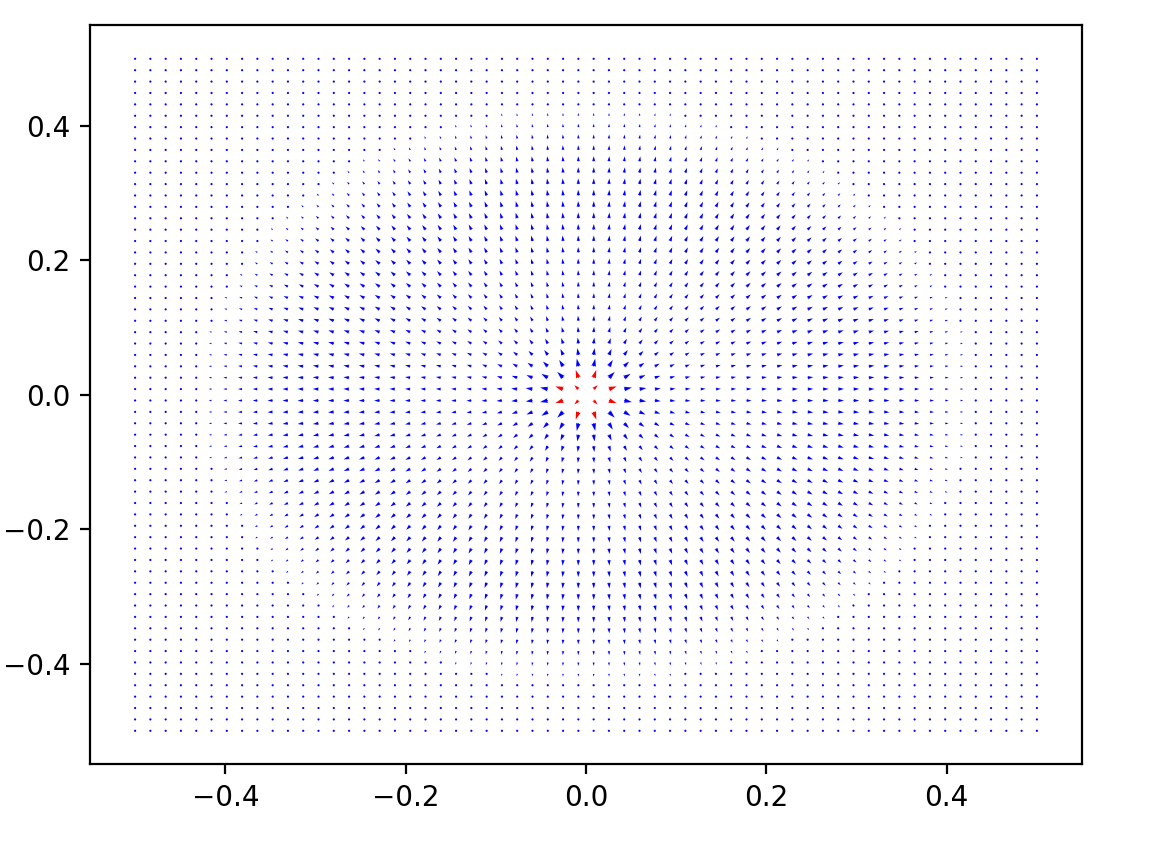}
              \caption{$t=0.0205$}
            \end{subfigure}

            \begin{subfigure}[b]{0.24\linewidth}
              \includegraphics[width=\linewidth]
                              {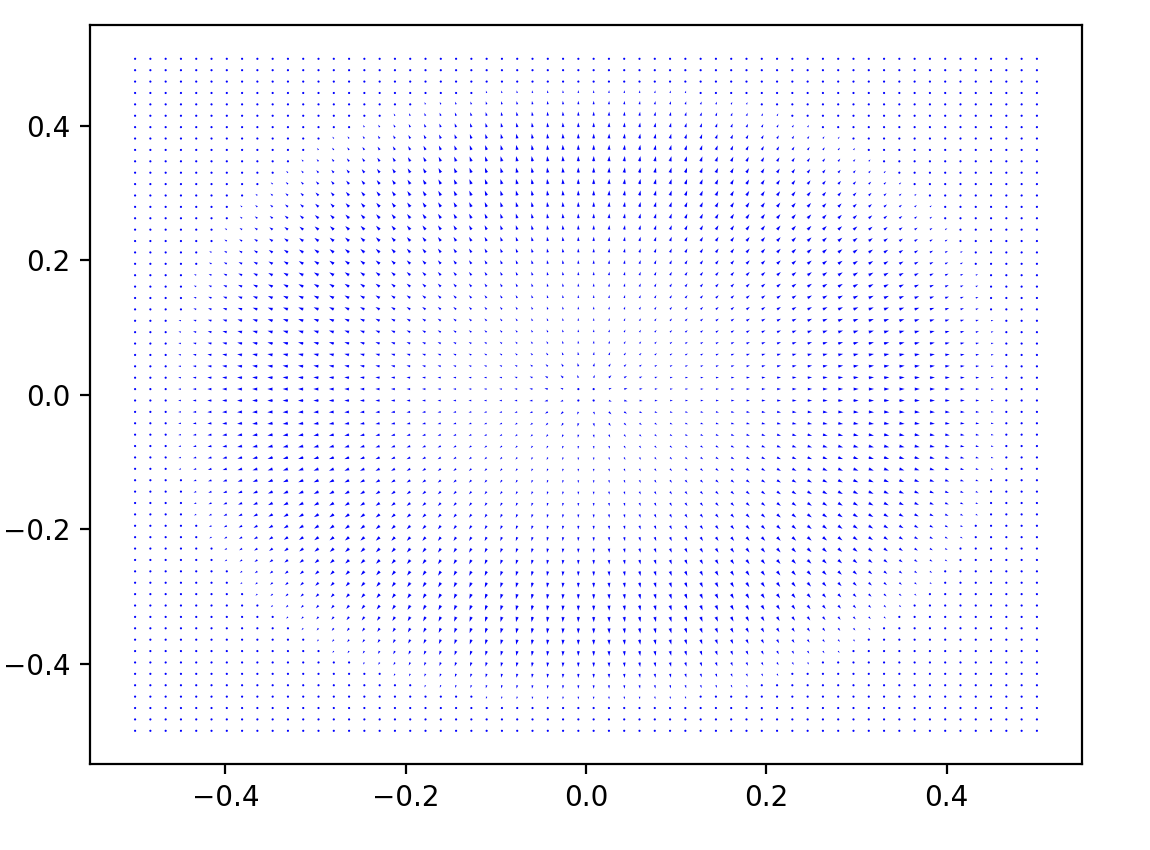}
              \caption{$t=0.254$}
            \end{subfigure}
            \begin{subfigure}[b]{0.24\linewidth}
              \includegraphics[width=\linewidth]
                              {t=0049.png}
              \caption{$t=0.049$}
            \end{subfigure}
            \begin{subfigure}[b]{0.24\linewidth}
              \includegraphics[width=\linewidth]
                              {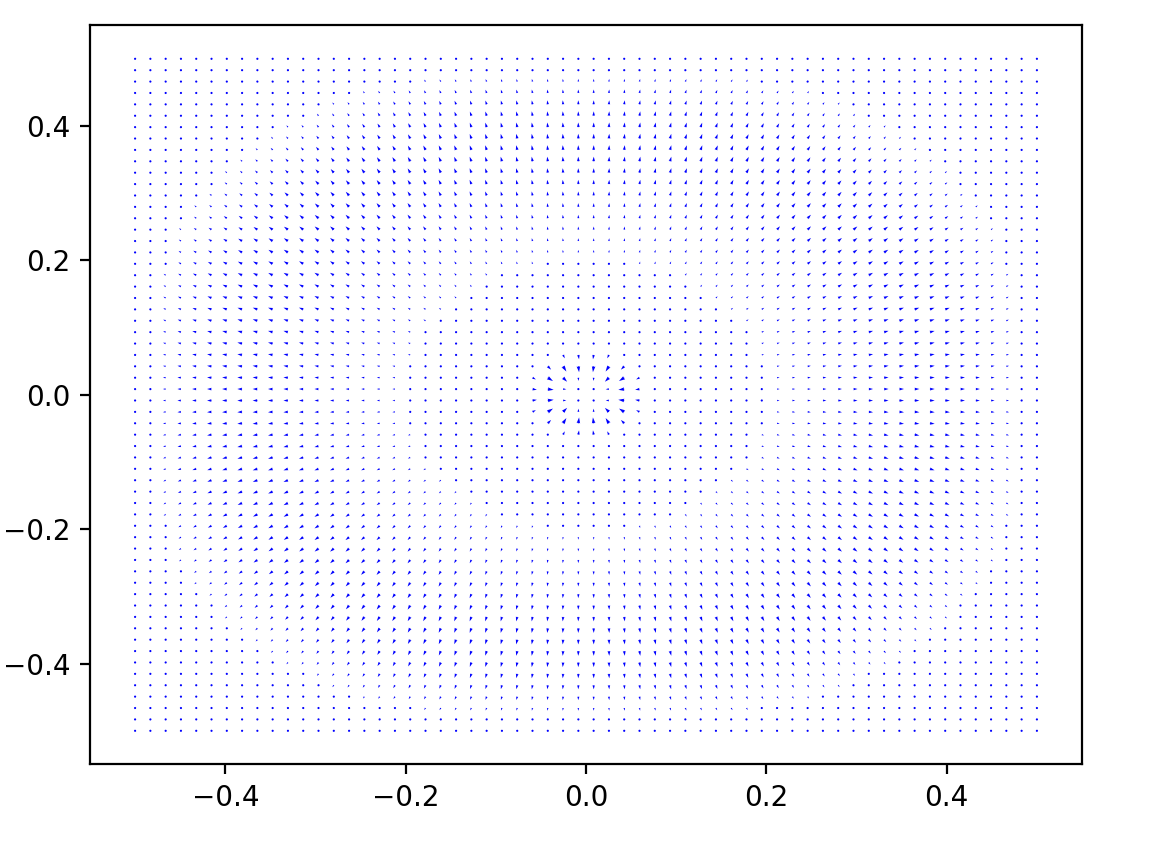}
              \caption{$t=0.303$}
            \end{subfigure}
            \begin{subfigure}[b]{0.24\linewidth}
              \includegraphics[width=\linewidth]
                              {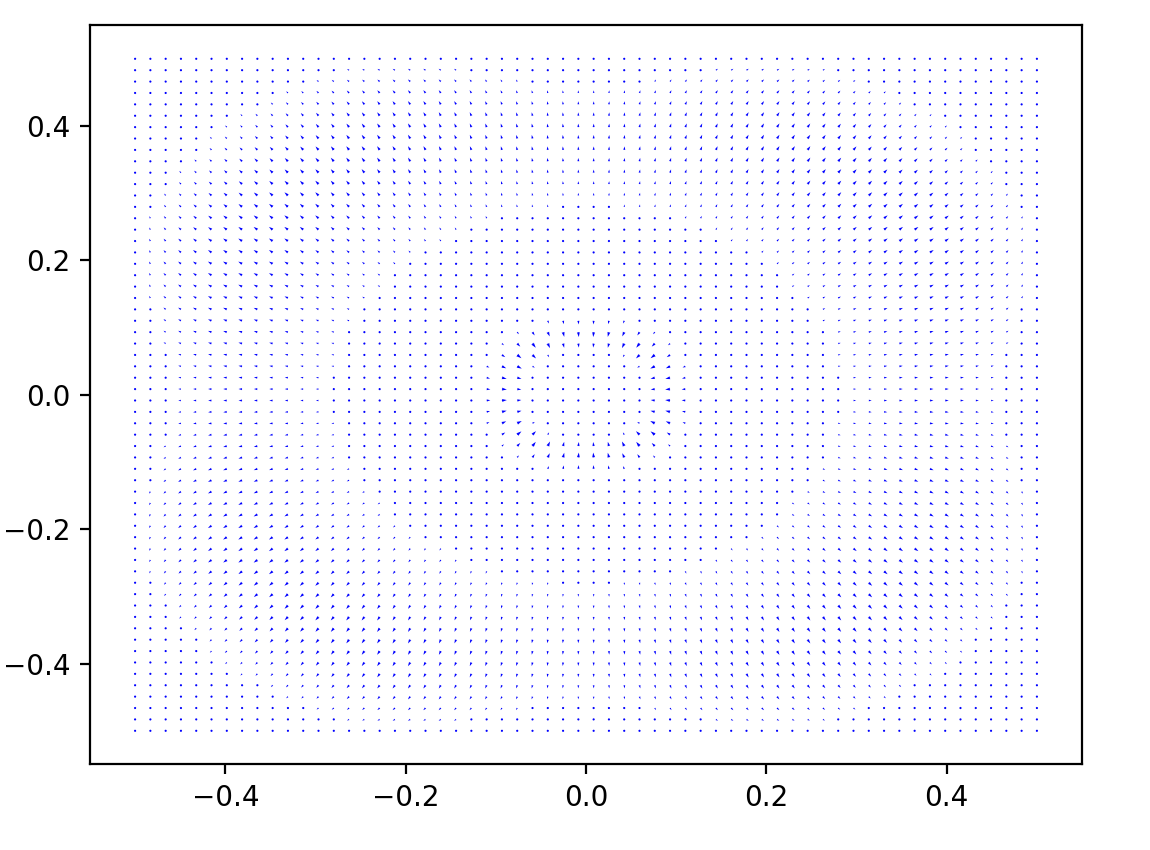}
              \caption{$t=0.352$}
            \end{subfigure}

          \end{center}
\end{figure}

Since the exact solution is unknown in this case, we use a numerical solution on a very fine mesh ($h = 2^{-9}, \tau=2^{-14}$) as reference solution
in order to approximate the error.
We observe that the error indeed scales as $\tau^2$, see Table~\ref{tab:erroreoc2}.  

\begin{table}
\caption{Error in energy norm and eoc up to time $\rT=0.2$. We use a mesh constant $h=2^{-9}$ and a reference time step $\tau = 2^{-14}$.}
\label{tab:erroreoc2}
 \begin{center}
\begin{tabular}{|c|c|c|c|c|}
\hline
$\tau$ & $ \| \omega_h- \omega_{\text{ref}}\|_{L^\infty(0,\rT;L^2(\Omega))} $ & EOC & $ \| \nabla u_h - \nabla u_{\text{ref}}\|_{L^\infty(0,\rT;L^2(\Omega))} $ & EOC  \\ \hline
$2^{-10}$ & 6.51e-04 & --- & 8.10e-04 & ---\\ \hline
$2^{-11}$& 2.13e-04 & 1.61 & 2.36e-04 & 1.78\\ \hline
$2^{-12}$ & 6.19e-05 & 1.78 & 6.51e-05 & 1.86\\ \hline
$2^{-13}$ & 1.28e-05 & 2.28 & 1.34e-05 & 2.28\\ \hline
 \end{tabular}
\end{center}
\end{table}

One of the goals we had in constructing the error estimator, 
i.e.\ the upper bound provided by combining Theorem \ref{thm:stab} and Lemma \ref{lem:compbounds}, was
that it is formally of optimal order, i.e.\ that it converges to zero with the same rate as the true error on equidistant meshes.
This is guaranteed as soon as the quantities 
$A^u, A^u_x, A^u_{xx}, A^\omega , A^\omega_x$, $B^u, B^u_x, B^u_{xx}, B^\omega , B^\omega_x$ are 
$\mathcal{O}(\tau)$ and $ C^u_x, C^u_{xx}, C^\omega , C^\omega_x$ are uniformly bounded in $\tau$.
This is indeed what we observe in the experiments we carried out. 
We do not report these numbers in detail but plot $\int_0^t \hat \alpha(s)\, ds$ and $\int_0^t \hat \delta(s)\, ds$ for several equidistant step-sizes $\tau$ and for different spatial mesh widths $h$, in Figure \ref{fig:ad}. Here $\hat \alpha, \hat \delta$ are computable upper bounds for $\alpha$ and $\delta$ based on  Lemma \ref{lem:compbounds}.
These results show that $\alpha$ and $\delta$ are independent of the spatial mesh width $h$, as they are supposed to be since we are quantifying time discretisation errors. We also observe that $\alpha$ is proportional to $\tau^2$ and $\delta$ is (essentially) independent of $\tau$. This confirms our theoretical predictions, i.e., optimal convergence order of the error estimator in the smooth regime.
  \begin{figure}[ht] 
\includegraphics[width=.5\textwidth]{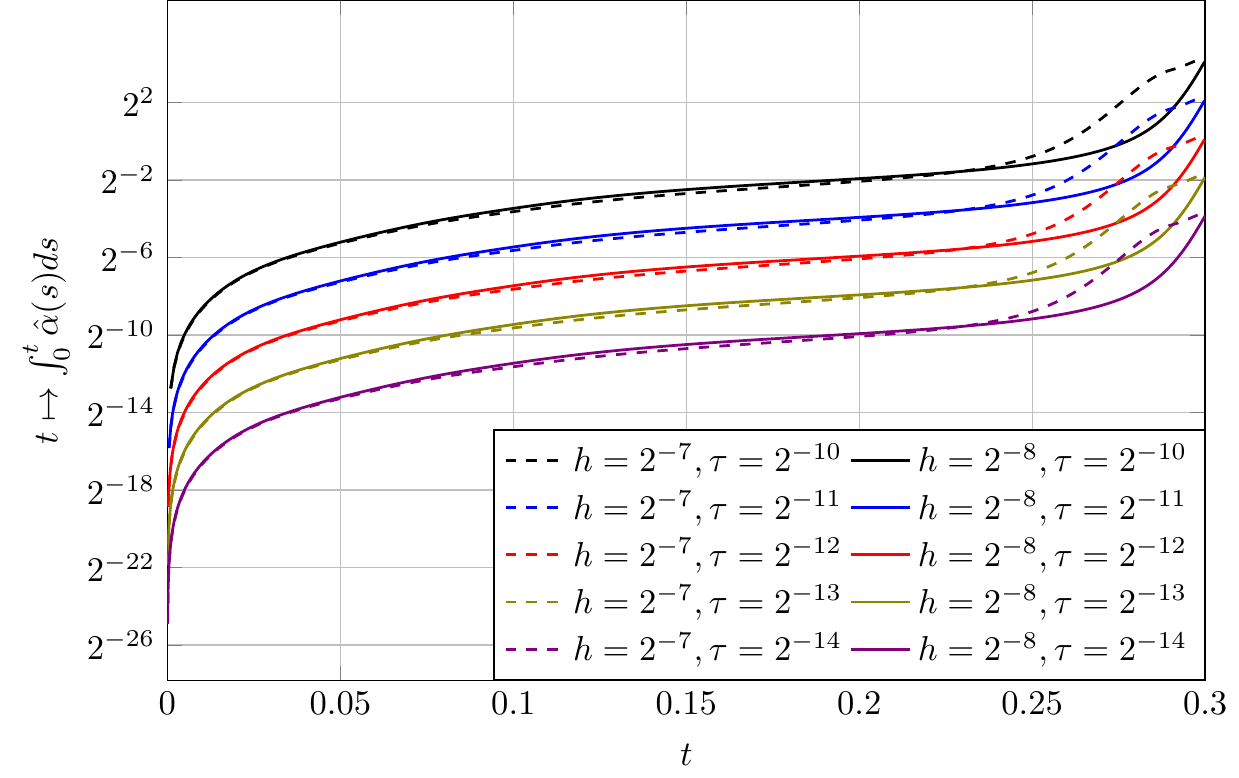}
\includegraphics[width=.5\textwidth]{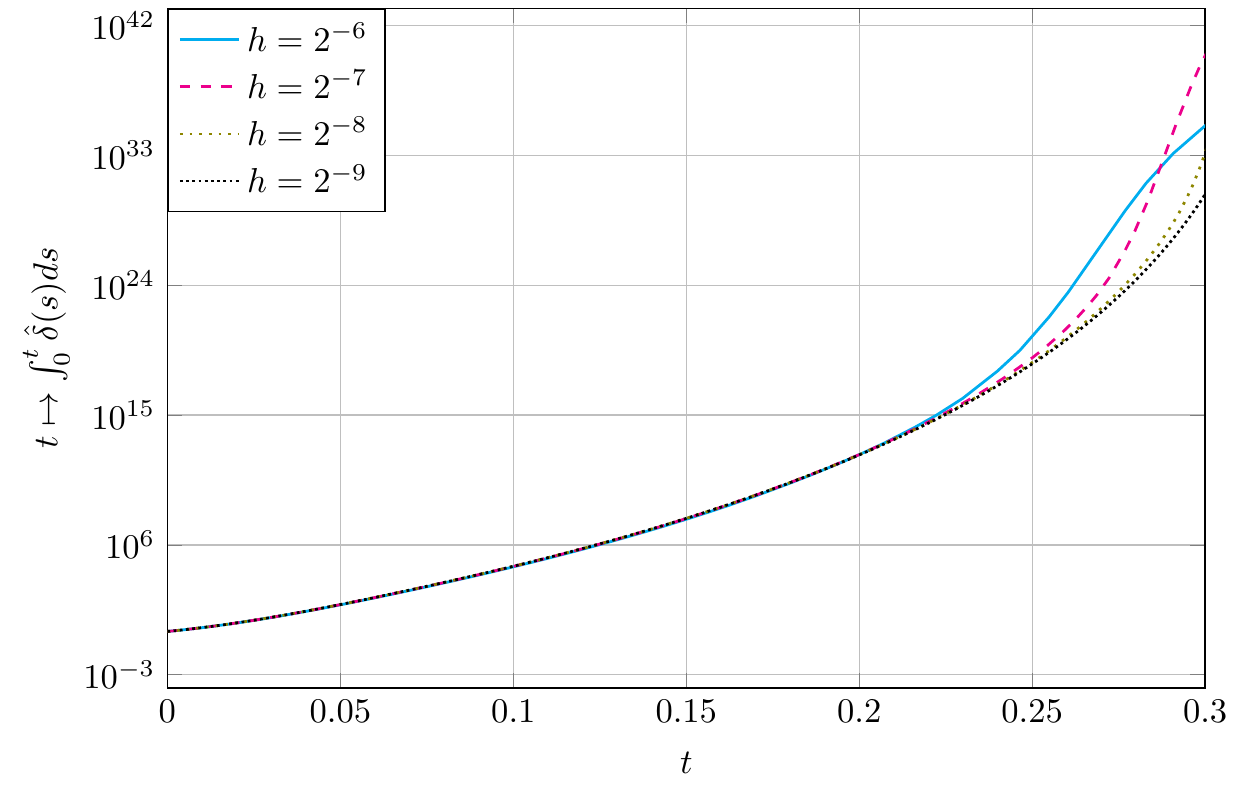}
\caption{Left:  Temporal evolution of the integral of computable upper bounds for $\alpha$ based on  Lemma \ref{lem:compbounds}. Note that, as long as the numerical solutions is ``smooth'', $\alpha$ is independent of the spatial mesh width $h$ and proportional to $\tau^2$. Once numerical solutions show a singularity, $\alpha$ does depend on $h$. Right:  Temporal evolution of the integral of computable upper bounds for $\delta$ based on  Lemma \ref{lem:compbounds}. Note that the bounds for $\delta$ are independent of $h$ and $\tau$ as long as the numerical solutions are ``smooth'', whereas $\delta$ is $h$ dependent once the numerical solutions display singularities. It should be noted that $\delta$ contains finite difference approximations of $\|\nabla u\|_{L^{2p}}$ and it is anticipated that if the gradient of the exact solution blows up the finite differences of the numerical solution approximating $\|\nabla u\|_\infty$ scale like $2/h$.}
\label{fig:ad}
 \end{figure} 

Due to the expected non-uniqueness of finite energy weak solutions we focus on time-stepping before and up-to singularity formation.
Standard algorithms for mesh adaptation aim at equi-distributing 
$\frac{1}{\tau_j} \int_{t_j}^{t_j + \tau_j}  \hat \alpha(s)\, ds$ amongst all time steps, i.e. the local step size is reduced if $\frac{1}{\tau_j} \int_{t_j}^{t_j + \tau_j} \hat \alpha(s)\, ds$ exceeds a given tolerance. 
As observed in \cite{Cangiani_Georgoulis_Kyza_Metcalfe_2016}, this strategy leads to excessive over-refinement close to blow-ups.
Moreover, 
equation \eqref{eq:stab_glo} shows that errors incurred in different time-steps are amplified by different factors. In order to obtain a scheme that is efficient up to blow-up, this needs to be taken into account and we follow a strategy that is similar to the one proposed in  \cite{Cangiani_Georgoulis_Kyza_Metcalfe_2016}, i.e. from the $j$-th to the $(j+1)$-th time step we increase the tolerance by a factor of $\exp(\tfrac 12\int_{t_{j-1}}^{t_j} \hat \delta(s) \, ds)$.

In Figure \ref{fig:adapt}, we compare values of the error estimator and time step sizes for the two time-stepping strategies described above.
 \begin{figure}[ht] 
\includegraphics[width=.5\textwidth]{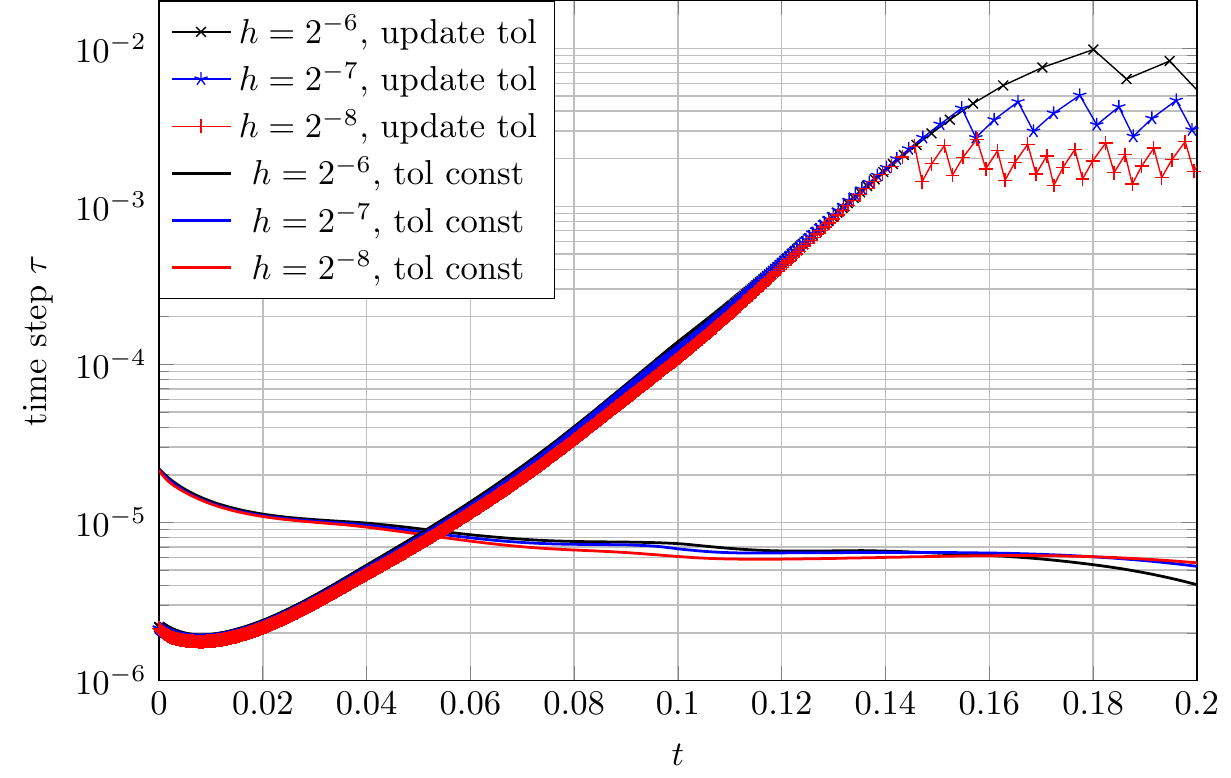}\includegraphics[width=.5\textwidth]{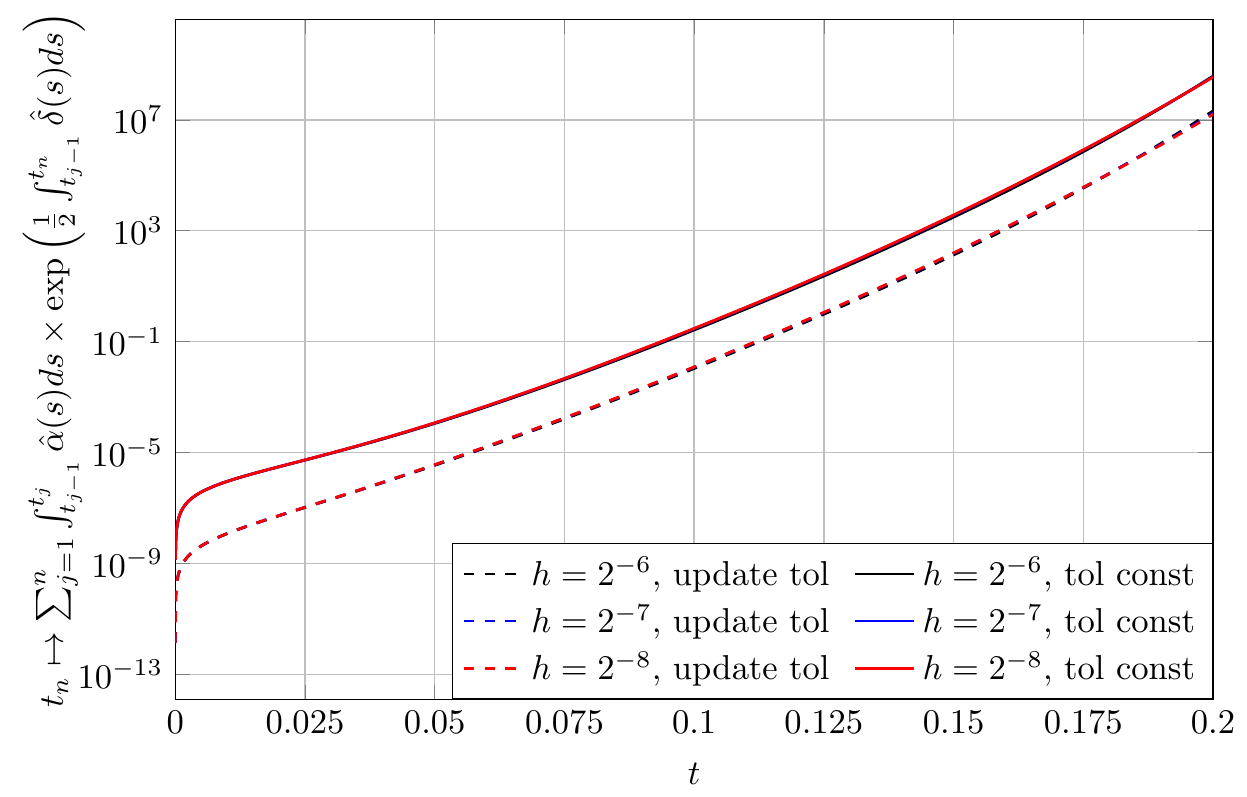}
\caption{Time step sizes (left) and values of the combined error estimator from Theorem \ref{thm:stab} and Lemma \ref{lem:compbounds} (right) as functions of time for different (equidistant) spatial meshes. We compare simulations with time step adaptation using a fixed tolerance $10^{-4}$ to ones using updated tolerances starting at $10^{-6}$, so that at time $t=0.17$ the number of computed time steps is comparable (approx. $21000$).}
\label{fig:adapt}
 \end{figure} 
 It turns out that the scheme using the updated error tolerances uses much smaller step sizes in the beginning and larger step sizes at later times than the equidistribution strategy. 
 We observe that at time $t=0.17$ the strategy using updated error tolerances has used fewer time steps but leads to a smaller error bound. 
 
We observe in Figure \ref{fig:adapt} that for time stepping with updated tolerances time step sizes stop increasing beyond a certain point in time. This is due to the fact that for time step sizes larger than some threshold value $\tau_{\text{thresh}} (h)>0$ the fixed point iteration that we use for solving the non-linear system \eqref{scheme} in each time step, see \cite[Algo. 2]{Bartels_2015} for details, does not converge. This results in the time step being automatically reduced independently of the error tolerance.

%

 \section*{Appendix} 
 This appendix is devoted to the proof of Lemma \ref{lem:compbounds}.
 Here, each quantity that is defined continuously in time is to be understood as its restriction to $(t_n,t_{n+1})$.
 Let us begin by giving explicit formulae for certain parts of $r_u$ and $r_\omega$. The first such expressions make explicit the difference between some piecewise quadratic, globally continuous function and its piecewise linear interpolations:
\begin{equation}\label{eq:maitri}
 \begin{aligned}
  u^*(t) - \widehat u(t) &= \frac 1 2 \frac{(t-t_n)(t_{n+1}-t)}{t_{n+1}-t_n} \left( u^{n+1} \times \omega^{n+1} - u^n \times \omega^n\right),\\
  \tw(t) - \widehat \omega(t) &= \frac 1 2 \frac{(t-t_n)(t_{n+1}-t)}{t_{n+1}-t_n} \left( \Delta u^{n+1} \times u^{n+1} - \Delta u^n \times u^n\right),\\
  \widehat u(t) \times \widehat w(t) - I_1(\widehat u \times \widehat w)(t) &= - \frac{(t-t_n)(t_{n+1}-t)}{(t_{n+1}-t_n)^2}  (u^{n+1} - u^n) \times (\omega^{n+1} - \omega^n).
 \end{aligned}
\end{equation}

\subsection*{Controlling $|\widehat u|$, $| u^*|$}
Let us now study how far away $\widehat u, u^*$ are from maps into the sphere:
If  $u^{n+1}-u^n$ is sufficiently small, a geometric argument implies 
\begin{equation*}
\Big| 1-|\widehat u| \Big| \leq |u^{n+1}-u^n|^2= (A^u)^2, \qquad \Big| 1-|u^*| \Big| 
\leq \Big| 1- |\widehat u| \Big|  + \Big| \widehat u - u^*\Big|  \leq  (A^u)^2 + \tau_n B^u
\end{equation*}
and, therefore,
\begin{equation}\label{eq:tuus}
 \Big|\tu - u^*\Big| =  \Big|\tu \Big(1-\Big|u^*\Big|\Big)\Big| \leq |u^{n+1}-u^n|^2= (A^u)^2+ \tau_n B^u.
\end{equation}
Thus, the conditions in Lemma \ref{lem:compbounds} imply
\begin{equation}\label{absvals}
 \frac 3 4 \leq |\widehat u| \leq \frac 54 , \qquad  \frac 12 \leq |u^*| \leq \frac 32.
\end{equation}

\subsection*{Estimating $r_{u,1}$}
Since $\tu(t_n)=\widehat u(t_n)$ and $\tw(t_n)=\widehat \omega(t_n) $ we  may rewrite $r_{u,1} $ as
\begin{equation}\label{eq:ru1} r_{u,1} = \tilde u \times (\tw - \widehat \omega) + (\tu - u^*) \times \widehat \omega + ( u^*- \widehat u) \times \widehat \omega   + \widehat u \times \widehat w - I_1(\widehat u \times \widehat w)\end{equation}
such that
\begin{multline} \Big| r_{u,1} \Big| \leq \tau_n (\Delta \tu^{n+1} \times \tu^{n+1} - \Delta \tu^{n} \times \tu^{n} )
2  |u^{n+1} - u^n|\, |\omega^{n+1} - \omega^n| \\ + \max\{|\omega^n|,|\omega^{n+1}|\} \left[|u^{n+1}-u^n|^2+ |t_{n+1}-t_n|\ |u^{n+1} \times \omega^{n+1} - u^n \times \omega^n|\right]\,.
\end{multline}
This proves \eqref{bound:ru1}.

We obtain \eqref{bound:nru1} by applying the product rule to \eqref{eq:ru1} since $|\widehat u(t,x)|\leq 1$.
We also use the fact that
 $\nabla \tu$ is the projection of $\nabla u^*$ onto the tangent space of the sphere, so that,  provided $|u^{n+1}-u^n|< 1/2$ holds, we have the point-wise estimate
\begin{equation}\label{eq:pwgrad} |\nabla \tu| \leq 2 |\nabla u^*|.  \end{equation}                                                                  

 \subsection*{Estimating $r_{u,3}$}
 The key to estimating $r_{u,3}$ is to control $\partial_t u^* \cdot u^*$. We notice that
 \begin{equation}\label{ptuu}
  \partial_t u^* \cdot u^* = (I_1[\widehat u \times \widehat \omega] + a_u) \cdot (\widehat u + (u^* - \widehat u)) 
 \end{equation}
and, thus, 
 \begin{equation}\label{ptuu2}
 | \partial_t u^* \cdot u^* | \leq  | I_1[\widehat u \times \widehat \omega]\cdot \widehat u |  + A^u A^\omega (1 + \tau_n B^u) + C^\omega  \tau_n B^u
 \end{equation}
 so that it remains to understand  $I_1[\widehat u \times \widehat \omega]\cdot \widehat u$. Using orthogonality, we obtain
 \begin{align}\begin{split}\label{ptuu3}
  I_1[\widehat u \times \widehat \omega]\cdot \widehat u &= [ \ell_n^0(t) u^n\times \omega^n + \ell_n^1(t) u^{n+1}\times \omega^{n+1} ]  \cdot [ \ell_n^0(t) u^n + \ell_n^1(t) u^{n+1} ] \\
  &=  \ell_n^0(t) \ell_n^1(t) [ u^{n+1} \cdot (u^n\times \omega^n ) + u^{n} \cdot (u^{n+1}\times \omega^{n+1} )]\\ 
  &=  \ell_n^0(t) \ell_n^1(t)\det(u^{n+1},u^n,\omega^{n+1}-\omega^n)\\
  &=  \ell_n^0(t) \ell_n^1(t)\det(u^{n+1}-u^n,u^n,\omega^{n+1}-\omega^n).\end{split}
 \end{align}
 We insert \eqref{ptuu3} into \eqref{ptuu} and obtain
  \begin{equation}\label{ptuu4}
 | \partial_t u^* \cdot u^* | \leq   A^u A^\omega (2 + \tau_n B^u) + C^\omega  \tau_n B^u\,.
 \end{equation}
Moreover, due to   
\eqref{absvals} we arrive at
\begin{equation}\label{ru31}
 \left| 1 - \frac 1 {|u^*|} \right| =  \left| 1 - \frac 1 {|\widehat u|}  - \frac{ | u^* |- | \widehat u|} { | u^* | | \widehat u|}\right| \leq \frac 43 (A^u)^2 +\frac 83 \tau_n B^u
\end{equation}
Using \eqref{ru31} and \eqref{ptuu4} we obtain:
\begin{equation}
|r_{u,3}| \leq (C^\omega + \frac 14 A^u A^\omega) \left(  \frac 43 (A^u)^2 +\frac 83 \tau_n B^u\right) + 4 A^u A^\omega (2 + \tau_n B^u) +4 C^\omega  \tau_n B^u.
\end{equation}
Let us note that for 
\begin{multline}\label{eq:nru31}
 \nabla r_{u,3}= \nabla \partial_t u^*  \left( 1 - \frac 1 {|u^*|}\right) + \nabla (I_1[\widehat u\times \widehat \omega] \cdot \widehat u) + \nabla a_u \cdot u^* \\
 +a_u \cdot \nabla u^* + \nabla (I_1[\widehat u\times \widehat \omega] \cdot (u^* - \widehat u))
 +  \partial_t u^*  \nabla \left( 1 - \frac 1 {|u^*|}\right)
\end{multline}
we have suitable bounds for all terms on the right hand side of \eqref{eq:nru31} except for $\nabla \left( 1 - \frac 1 {|u^*|}\right)$, e.g.,  $\nabla (I_1[\widehat u\times \widehat \omega] \cdot  \widehat u)$ can be estimated by applying the product rule to \eqref{ptuu3}.
As a first step towards estimating for $\nabla \left( 1 - \frac 1 {|u^*|}\right)= - \frac{u^* \cdot \nabla u^*}{|u^*|^3}$ we compute
\begin{equation}\label{eq:nru32}
\partial_{x_j}\! \left( 1 - \frac 1 {|u^*|}\right)=\frac{ \partial_{x_j} \widehat u \cdot \widehat u + \partial_{x_j} (u^* - \widehat u) \cdot \widehat u + \partial_{x_j} \widehat u\cdot\! (u^* - \widehat u) + \partial_{x_j} (u^*- \widehat u)\cdot (u^* - \widehat u)}{|u^*|^3}.
\end{equation}
We recall $|u^n|=1$, which implies $\partial_{x_j} u^n \cdot u^n=0$, for all $n$, so that
\begin{multline}\label{eq:nru33}
\partial_{x_j} \widehat u \cdot \widehat u = \partial_{x_j} (\ell_n^0 (t) u^n + \ell_n^1 (t) u^{n+1}) \cdot (\ell_n^0 (t) u^n + \ell_n^1 (t) u^{n+1})\\ = \ell_n^0 (t)\ell_n^1(t) ( \partial_{x_j} u^n \cdot  u^{n+1} + \partial_{x_j} u^{n+1} \cdot  u^{n} )
 = - \ell_n^0 (t)\ell_n^1(t) \partial_{x_j} ( u^{n+1} - u^n )\cdot  ( u^{n+1} - u^n ).
\end{multline}
We insert \eqref{eq:nru33} into \eqref{eq:nru32} and obtain
\begin{equation}\label{eq:nru34}
\left|\nabla \left( 1 - \frac 1 {|u^*|}\right)\right| \leq 8[ A^u_x A^u + \tau_n B^u_x + C^u_x \tau_n B^u + \tau_n^2 B^u_x B^u]
\end{equation}
Thus, we obtain
\begin{multline}\label{eq:nru35}
| \nabla r_{u,3}| \leq 
 \left( C^u_x C^\omega + C^\omega_x + \frac 14 A^u_x A^\omega +\frac 14 A^u A^\omega_x\right) \left(\frac 43 (A^u)^2 +\frac 83 \tau_n B^u\right)\\ + \frac 32 A^u_x A^\omega +2 A^u C^u_x A^\omega + \frac 32 A^u A^\omega_x  
+ \left(C^u_x C^\omega + C^\omega_x \right) \tau_n B^u + C^\omega \tau_n B^u_x\\ + 8\left(C^\omega + \frac 14 A^u A^\omega\right) \left( A^u_x A^u + \tau_n B^u_x + C^u_x \tau_n B^u + \tau_n^2 B^u_x B^u\right).
\end{multline}
 This completes providing bounds for the different components of $r_u$ and $\nabla r_u$.
 
 \subsection*{Estimating $r_\omega$}
 Obviously,
 $ |r_{\omega,2}|= |a^\omega| \leq \frac 14 A^u A^u_{xx}$
 and 
 \begin{equation}\label{eq:karuna} 
  r_{\omega,1} = (\Delta \tu - \Delta u^*)\times \tu + (\Delta u^* - \Delta \widehat u)\times \tu + \Delta \widehat u \times (\tu - u^*) + \Delta \widehat u \times (u^*-\widehat u) + \Delta \widehat u \times \widehat u - I_1[\Delta \widehat u \times \widehat u].
 \end{equation}
We insert \eqref{eq:tuus} and \eqref{eq:maitri} into \eqref{eq:karuna} and obtain
 \begin{equation}\label{eq:mudito} 
  |r_{\omega,1}| \leq |(\Delta \tu - \Delta u^*) | + \tau_n B^u  + C^u_{xx} (A^u)^2 + C^u_{xx}\tau_n B^u + A^u_{xx} A^u\,,
 \end{equation}
 where we have used that
\[ \Delta \widehat u \times \widehat u - I_1[\Delta \widehat u \times \widehat u] =- \ell^0_n(t)\ell^1_n(t)  (\Delta u^{n+1}-\Delta u^n) \times (u^{n+1} - u^n).\]
It remains to provide an estimate for $|\Delta \tu - \Delta u^*|$.
We note that $u^*=\tu |u^*|$ and, thus,
\begin{align}
 \Delta (\tu - u^*) = \Delta u^* \left( 1 - \frac {1}{|u^*|} \right) - \frac{\sum_j \partial_{x_j} \tu \cdot \partial_{x_j} |u^*| }{ |u^*| } - \frac{ \tu \Delta  |u^*| }{ |u^*| }, 
\end{align}
so that
\begin{align}
 \begin{split}
\label{eq:623}
| \Delta (\tu - u^*) | &\leq   ( C^u_{xx}+ \tau_n B^u_{xx}) \left( \frac 43 (A^u)^2+ \frac 83 \tau_n B^u\right)  + |\nabla \tu|  \frac{|u^* \cdot \nabla u^*| }{ |u^*|^2 } + \frac{ |\Delta  |u^*| \, | }{ |u^*| }\\
&\leq  ( C^u_{xx}+ \tau_n B^u_{xx}) \left( \frac 43 (A^u)^2+ \frac 83 \tau_n B^u\right)
\\
&+(C^u_x + \tau_n B^u_x)2[ A^u_x A^u + \tau_n B^u_x + C^u_x \tau_n B^u + \tau_n^2 B^u_x B^u]
+ \frac{ |\Delta  |u^*| \, | }{ |u^*| }
\end{split}
\end{align}
where we have used \eqref{eq:nru34}.
Orthogonality $\partial_{x_j} u^n \perp u^n$  for all $j$ and all $n$ implies
\begin{equation}\label{eq:624}
 \Delta  |u^*| = - \ell_n^0 (t)\ell_n^1(t)[(u^{n+1} -u^n)\cdot \partial_{x_j}^2 (u^{n+1} -u^n) + |\partial_{x_j} (u^{n+1} -u^n)|^2] .
\end{equation}
Inserting \eqref{eq:624} into \eqref{eq:623} implies
\begin{multline}\label{eq:625}
| \Delta (\tu - u^*) |
\leq  ( C^u_{xx}+ \tau_n B^u_{xx}) \left( \frac 43 (A^u)^2+ \frac 83 \tau_n B^u\right)  + A^u_{xx} A^u + (A^u_x)^2\\
+(C^u_x + \tau_n B^u_x)2[ A^u_x A^u + \tau_n B^u_x + C^u_x \tau_n B^u + \tau_n^2 B^u_x B^u]
\end{multline}
Inserting \eqref{eq:625} into \eqref{eq:mudito} completes the bound for $r_\omega$.

\subsection*{Estimating $r_g$}
We note that orthogonality $u^n \perp \omega^n$ implies
\begin{multline}
 \tu \cdot \tw = \tu\cdot(\tw - \widehat \omega ) - (\tu - \widehat u)\cdot \widehat \omega  + \widehat u \cdot \widehat \omega \\
 =  \tu\cdot(\tw - \widehat \omega ) - (\tu - \widehat u)\cdot \widehat \omega  - \ell_n^0 (t)\ell_n^1(t)(u^{n+1} -u^n)\cdot (\omega^{n+1} -\omega^n)
\end{multline}
so that
\begin{equation}
 |  \tu \cdot \tw| \leq \tau_n B^\omega + C^\omega( (A^u)^2 + \tau_n B^u) + A^u A^\omega.
\end{equation}
Thus, using the definition of $r_g$
\begin{align*}
 | r_g | &\leq (C^\omega + \tau_n B^\omega)) [\tau_n B^\omega + C^\omega( (A^u)^2 + \tau_n B^u) + A^u A^\omega] \\
 &\qquad\qquad + [\tau_n B^\omega + C^\omega( (A^u)^2 + \tau_n B^u) + A^u A^\omega]^2.
\end{align*}

 \bibliographystyle{siamplain}
 \bibliography{wavemaps}
   
\end{document}